\documentclass[preprint,11pt]{elsarticle}

\usepackage{amsfonts,amsthm,amsmath}
\usepackage{mathrsfs}
\usepackage{vaucanson-g}
\usepackage{gastex}

\newcommand{\sh}{\mathscr{S}}
\def\D{\{0,1\}}

\newcommand{\llex}{\le_{\rm lex}}
\newcommand{\sllex}{<_{\rm lex}}
\newcommand{\lex}{{\rm lex}}
\def\A{\mathbb{A}}
\def\B{\mathbb{B}}
\newcommand{\nats}{{\mathbb N}}
\newcommand{\reals}{{\mathbb R}}
\newcommand{\rats}{{\mathbb Q}}

\newcommand{\pref}{{\rm pref}}

\newtheorem{thm}{Theorem}
\newtheorem{ques}[thm]{Question}
\newtheorem{theorem}{Theorem}[section]
\newtheorem{lemma}[theorem]{Lemma}
\newtheorem{corollary}[theorem]{Corollary}
\newtheorem{proposition}[theorem]{Proposition}
\newtheorem{question}[theorem]{Question}
\theoremstyle{definition}
\newtheorem{remark}[theorem]{Remark}

\newtheorem{definition}[theorem]{Definition}

\begin{document}

\begin{frontmatter}

\title{Infinite Self-Shuffling Words}

\author[label3]{\'Emilie Charlier}
  \ead{echarlier@ulg.ac.be}

  \author[label4]{Teturo Kamae}
  \ead{kamae@apost.plala.or.jp}

 \author[label5,label6]{Svetlana Puzynina\fnref{label1}}
  \ead{svepuz@utu.fi}

   \author[label5,label7]{Luca Q. Zamboni\fnref{label2}}
  \ead{lupastis@gmail.com}

  \fntext[label1]{Partially supported by the Academy of Finland under grant 251371,
by Russian Foundation of Basic Research (grants 12-01-00448 and
12-01-00089).}
  \fntext[label2]{Partially supported by a FiDiPro grant (137991) from the Academy of Finland and by
ANR grant {\sl SUBTILE}.}
\address[label3]{D\'epartement de Math\'ematique, Universit\'e de Li\`ege, Belgium}
\address[label4]{Advanced Mathematical Institute, Osaka City University,  Japan}
\address[label5]{FUNDIM, University of Turku, Finland}
\address[label6]{Sobolev Institute of Mathematics, Novosibirsk Russia}
\address[label7]{Universit\'e de Lyon,
Universit\'e Lyon 1, CNRS UMR 5208,
Institut Camille Jordan,
43 boulevard du 11 novembre 1918,
F69622 Villeurbanne Cedex, France}

\begin{abstract}
In this paper we introduce and study a new property of infinite words:
An infinite word $x\in \A^\nats$, with values in a finite set $\A$,
is said to be $k$-{\it self-shuffling} $(k\geq 2)$ if $x$ admits factorizations:
$x=\prod_{i=0}^\infty U_i^{(1)}\cdots U_i^{(k)}=\prod_{i=0}^\infty U_i^{(1)}=\cdots =\prod_{i=0}^\infty U_i^{(k)}$.
In other words, there exists a shuffle of $k$-copies of $x$ which produces $x$.
We are particularly interested in the case $k=2$, in which case we say $x$ is self-shuffling.
This property of infinite words is shown to be independent of the complexity of the word
as measured by the number of distinct factors of each length.
Examples exist from bounded to full complexity.
It is also an intrinsic property of the word and not of its language (set of factors).
For instance, 
every aperiodic word contains a non self-shuffling word in its
shift orbit closure. While the property of being self-shuffling is
a relatively strong condition, many important words arising in the
area of symbolic dynamics are verified to be self-shuffling. They
include for instance the Thue-Morse word fixed by the morphism
$0\mapsto 01$, $1\mapsto 10$.
As another example we show that all
Sturmian words of slope $\alpha \in \reals \setminus \rats$ and
intercept $0<\rho <1$ are self-shuffling (while those of intercept
$\rho=0$ are not). Our characterization of self-shuffling Sturmian
words can be interpreted arithmetically in terms of a dynamical
embedding and defines an arithmetic process we call the {\it
stepping stone model}. One important feature of self-shuffling
words stems from its morphic invariance: The morphic image of a
self-shuffling word is  self-shuffling. This provides a useful
tool for showing that one word is not the morphic image of
another. In addition to its morphic invariance, this new notion
has other unexpected applications particularly in the area of
substitutive dynamical systems. For example, as a consequence of
our characterization of self-shuffling Sturmian words, we recover
a number theoretic result, originally due to Yasutomi, on a
classification of pure morphic Sturmian words in the orbit of the
characteristic.
\end{abstract}


\end{frontmatter}

\section{Introduction}

Let $\A$ be a finite non-empty set. We denote by $\A^*$ (resp. $\A^\nats)$ the set of
all finite (resp. infinite) words $u=x_0x_1x_2\cdots $ with $x_i\in \A$.

Given $k$ finite words $x^{(1)},x^{(2)}, \ldots ,x^{(k)} \in \A^*$
we let  $\sh (x^{(1)},x^{(2)}, \ldots ,x^{(k)})\subseteq \A^*$ denote the collection of all words $z$
for which there exists a factorization
\[
z=\prod_{i=0}^nU_i^{(1)}U_i^{(2)}\cdots U_i^{(k)}
\]
with each $U_i^{(j)} \in \A^*$ and with
$x^{(j)}=\prod_{i=0}^nU_i^{(j)}$ for each $1\leq j\leq k$.
Intuitively, $z$ may be obtained as a {\it shuffle} of the words $x^{(1)},x^{(2)}, \ldots ,x^{(k)}$.
For instance, it is readily checked that $011100110 \in \sh (0010,101,11)$.
Analogously, given $k$ infinite words $x^{(1)},x^{(2)}, \ldots ,x^{(k)} \in \A^\nats$
we define $\sh (x^{(1)},x^{(2)}, \ldots ,x^{(k)})\subseteq \A^\nats$ to be the collection of all infinite words $z$
for which there exists a factorization
\[
z=\prod_{i=0}^\infty U_i^{(1)}U_i^{(2)}\cdots U_i^{(k)}
\]
with each $U_i^{(j)} \in \A^*$ and with
$x^{(j)}=\prod_{i=0}^\infty U_i^{(j)}$ for each $1\leq j\leq k$.

Finite word shuffles were extensively studied in \cite{HRS}. Given
$x\in \A^*$, it is generally a difficult problem to determine
whether there exists $y\in \A^*$ such that $x\in \sh (y,y)$ (see
Open Problem 4 in \cite{HRS}). The problem has recently been shown
to be NP-complete for sufficiently large alphabets \cite{BS,RV}.
However, in the context of infinite words, this question is
essentially trivial: In fact, it is readily verified that if $x
\in \A^\nats$ and each symbol $a\in \A$ occurring in $x$ occurs an
infinite number of times in $x$, then there exists at least one
(and typically infinitely many) $y\in \A^\nats$ with $x\in \sh
(y,y)$. Instead, in the framework of infinite words, a far more
delicate question is the following:

\begin{ques}
Given $x\in \A^\nats$, does there exist an integer $k\geq 2$ such that $x \in \sh (\underbrace{x,x,\ldots ,x}_{k})$?
\end{ques}

\noindent If such a $k$ exists, we say $x$ is $k$-{\it self-shuffling}.
In case $k=2$, we say  $x$ is {\it self-shuffling}. It is not difficult to see
that every self-shuffling word  is $k$-self-shuffling for each $k\geq 2$.

If $x\in \A^\nats$ is $k$-self-shuffling, then there exists at
least one word $s\in \{1,2,\ldots ,k\}^{\nats}$ (called the {\it
steering word}) which defines the shuffle. Typically a
$k$-self-shuffling word $x$ can be shuffled in more than one way
so as to reproduce itself, i.e., may define more than one steering
word. In contrast, every word $s\in \{1,2,\ldots ,k\}^{\nats}$ is
a steering word for some $k$-self-shuffling word.
Moreover, if $s$ begins in a block of the form $a^rb$ with $a$ and $b$ distinct symbols in $\{1,2,\ldots ,k\}$,
then $s$ is the steering word of a unique (up to word isomorphism) non-constant $k$-self-shuffling word $x(s)$
on an alphabet of size $r$. 
In general, the relationship between properties of $s$ and $x(s)$ is a mystery.
For instance, it may be that $s$ is uniformly recurrent and $x(s)$
not. An infinite word $w$ is \emph{uniformly recurrent}, if for
each its factor $u$ there exists an integer $n$ such that each
factor of $w$ of length $n$ contains $u$ as its factor.
In this paper, 
we are primarily interested in $k$-self-shuffling words,
and less importance is placed on the corresponding steering words.
In fact, we mainly focus on self-shuffling words,
although many of the results presented here extend to general $k$.
Thus $x\in \A^\nats$ is self-shuffling if and only if $x$ admits factorizations
\[
x=\prod_{i=1}^\infty U_iV_i=\prod_{i=1}^\infty U_i=\prod_{i=1}^\infty V_i
\]
with $U_i,V_i \in \A^+$.

The simplest class of self-shuffling words consists of all
(purely) periodic words $x=u^\omega$. We note that if $x\in
\A^\nats$ is self-shuffling, then every letter $a\in \A$ occurring
in $x$ must occur an infinite number of times. Thus for instance,
the ultimately periodic word $01^\omega$ is not self-shuffling. On
the opposite extreme, we show the existence of self-shuffling
words having full complexity (i.e., they have all finite words
over a given alphabet as their factors). Thus, the property of
being self-shuffling is largely independent of the usual (subword)
"complexity" of an infinite word as measured by the number of
blocks of each given length. Moreover it is also an intrinsic
feature of the word and not of its language (or set of factors).
For instance, the {\it Fibonacci word}
\[
x=0100101001001010010100\cdots,
\]
defined as the fixed point of  the substitution $0\mapsto 01$,
$1\mapsto 0$, is self-shuffling while $0x$ and $1x$ are not (see
\S2). More generally, we will show that given any aperiodic word
$x\in \A^\nats$, the shift orbit closure of $x$ always contains at
least one point which is not self-shuffling.

While the property of being self-shuffling is quite strong,
many important words arising in symbolic dynamics turn out to be self-shuffling.
This includes  the famous {\it Thue-Morse} word
\[
\mathbf{T} =0110100110010110100101100110100110010110\cdots
\]
whose origins go back to the beginning of the last century
with the works of  the Norwegian mathematician Axel Thue \cite{Th1}.
The $n$th entry $t_n$ of $\mathbf{T}$ is defined as the sum modulo $2$ of the digits in the binary expansion of $n$.
The Thue-Morse word is linked to many different areas of mathematics:
from discrete mathematics to number theory to differential geometry (see for example \cite{AS1,AS2}).
While much is already known on the combinatorial properties of the Thue-Morse word,
proving that $\mathbf{T}$ is self-shuffling is less straightforward than expected.

Sturmian words constitute another important class of aperiodic self-shuffling words.
Sturmian words are infinite words having exactly $n+1$ factors of length
$n$ for each $n \geq 1$.  Their origins can be traced back to the astronomer J. Bernoulli  III in 1772.
Sturmian words arise naturally in different areas of mathematics including
combinatorics, algebra, number theory, ergodic theory, dynamical systems and differential equations.
They are also of great importance in theoretical physics (as basic  examples of $1$-dimensional quasicrystals)
and in theoretical computer science where they are used in
computer graphics as digital approximation of straight lines.
Sturmian words are regarded as the most basic  non-ultimately periodic infinite words.
Perhaps the most famous and well studied  Sturmian word is the Fibonacci word defined above.
In the 1940's,  Hedlund and Morse \cite{MoHe2} showed that each Sturmian word is the symbolic
coding of the orbit of a point $x$ (called the {\it intercept}) on the unit circle under a rotation
by an irrational angle $\alpha$ (called the {\it slope}), where the circle is  partitioned
into two complementary intervals, one of length $\alpha$ and the other of length $1-\alpha$.
Conversely each such coding defines a Sturmian word.
It is well known that the dynamical/ergodic properties of the system,
as well as the combinatorial properties of the associated Sturmian word,
hinge on the arithmetical/Diophantine qualities of the angle $\alpha$ given by its continued fraction expansion.
As in the case of the Fibonacci word, we show that for every irrational number $\alpha$,
all (uncountably many) Sturmian words of slope $\alpha$  and intercept $\rho$ are self-shuffling
except for the two Sturmian words corresponding to $\rho=0$.

In this paper, we derive a number of necessary (and in some cases
sufficient) conditions for a word to be self-shuffling. For
instance, if a word $x$ is self-shuffling, then $x$ begins in only
finitely many Abelian unbordered words. As an application of this
we show that the well-known {\it paper-folding word} is not
self-shuffling. Infinite {\it Lyndon words} \cite{SI} constitute
another class of words which are shown not to be self-shuffling. A
word $x\in \A^\nats$ is said to be Lyndon if there exists an order
on $\A$ with respect to which $x$ is lexicographically smaller
than each of its tails.
 We prove that if $x$ is Lyndon, then any $z \in \sh (x,x)$ is lexicographically smaller than $x$,
from which it follows immediately that $x$ is not self-shuffling.
While this may appear rather intuitive, our proof of this fact is
both long and delicate. We further prove that each aperiodic word
$x$ admits a Lyndon word in its shift orbit closure, i.e., there
exists a Lyndon word $y$ each of whose factors is a factor of $x$.
From this it follows that each aperiodic word $x$ admits a
non-self-shuffling word in its shift orbit closure.

An important feature of the self-shuffling property stems from its
invariance under the action of a morphism: The morphic image of a
self-shuffling word is again self-shuffling. Many important
classes of words (e.g., Sturmian words, pure morphic words, and
Toeplitz words) are not preserved by the action of an arbitrary
morphism. This invariance provides a useful tool for showing that
one word is not the morphic image of another. For instance, the
paper-folding word is not the morphic image of any self-shuffling
word. However this application requires knowing a priori whether a
given word is or is not self-shuffling. In general, to show that a
word is self-shuffling, one must actually exhibit a shuffle.
Self-shuffling words have other unexpected applications
particularly in the study of substitutive dynamical systems. For
instance, as an almost immediate consequence of our
characterization of self-shuffling Sturmian words, we recover a
result, originally proved by Yasutomi via number theoretic
methods, which gives a characterization of pure morphic Sturmian
words in the orbit of the characteristic.

The paper is organized as follows:
In \S2 we establish some general properties of $k$-self-shuffling words
and  along the way give various examples and non-examples.
Here we also consider self-shuffling words which are fixed points of primitive substitutions.
Under some additional assumptions on the substitution,
we deduce the self shuffling property for other points in the shift orbit of $x$.
For instance, if $\tau$ is a primitive substitution having a unique periodic point $x$,
then if $x$ is self-shuffling then the same is true of each shift of $x$.
In \S3 we establish the self-shuffling of the Thue-Morse word by explicitly constructing a shuffle.
Our proof makes use of different morphisms associated with the Thue-Morse word.
In \S4 we prove that Lyndon words are not self-shuffling.
In \S5 we obtain a  characterization of self-shuffling Sturmian words
and derive various applications including Yasutomi's result mentioned above.
In \S6 we give an arithmetic interpretation of our characterization
of self-shuffling Sturmian words in terms of a dynamical embedding
of an infinite graph into the dynamical system corresponding to a circle rotation.
In this framework we describe an arithmetic process we call the {\it stepping stone model}
which may be of independent interest in the theory of Diophantine approximations.
We end the paper with a few open questions.

A preliminary and incomplete version of this paper has been
reported at ICALP 2103 conference \cite{icalp}.

\section{General properties}\label{section_gen}

In this section we develop some basic properties of
$k$-self-shuffling words. Let $\A$ be a finite non-empty set. We
denote by $\A^*$ the set of all finite words $u=x_1x_2\cdots x_n$
with $x_i\in \A$. The quantity $n$ is called the length of $u$ and
is denoted $|u|$. For a letter $a\in \A$, let $|u|_a$  denote the
number of occurrences of $a$ in $u$. The empty word, denoted
$\varepsilon$, is the unique element in $\A^*$ with
$|\varepsilon|=0$. We set $\A^+=\A^*-\{\varepsilon\}$. We denote
by $\A^\nats$ the set of all one-sided infinite words
$x=x_0x_1x_2\cdots$ with $x_i\in \A$.  The \emph{shift} map is the
mapping defined by $x_0x_1x_2\cdots\mapsto x_1x_2x_3\cdots$. The
\emph{shift orbit closure} of an infinite word $x$ could be
defined, e.g., as the set of infinite words whose sets of factors
are included in the set of factors of $x$. Given
$x=x_0x_1x_2\cdots \in \A^\nats$ and a finite or infinite subset
$N=\{N_0<N_1<N_2<\cdots\}\subseteq\nats$, we put
$x[N]=x_{N_0}x_{N_1}x_{N_2}\cdots \in \A^\nats$.

\begin{definition}\label{k-self-shuffling}
Let $x\in \A^\nats$ and $k\in\{2,3,\ldots\}$. We say  $x$ is $k$-{\it self-shuffling}
if $x$ satisfies any one of the following two equivalent conditions:
\begin{itemize}
\item  $x \in \sh (\underbrace{x,x,\ldots ,x}_{k})$.
\item There exists a $k$-element partition of $\nats$ into infinite subsets $N^1,N^2,\ldots, N^k$ with
 $x[N^i]=x$ for each $i=1,\ldots,k$.
  \end{itemize}
\end{definition}

\noindent  In case  $x$ is $2$-self-shuffling we say simply that $x$ is {\it self-shuffling}.
It is evident that if $x$ is self-shuffling, then $x$ is $k$-self-shuffling for each $k\geq 2$.
Later we give an example of a $3$-self-shuffling word which is not self-shuffling.

If $x$ is $k$-self-shuffling, then there exists a word $s$ on a $k$-letter alphabet which defines or steers the shuffle.
We call such a word a {\it steering word} for the shuffle. More precisely, if $x\in \A^\nats$ is $k$-self-shuffling,
then there exists a $k$-element partition of $\nats$ into infinite subsets $N^1,N^2,\ldots, N^k$ with
 $x[N^i]=x$ for each $i=1,\ldots,k$.
The corresponding steering word $s=s_0s_1s_2\cdots \in
\{1,2,\ldots ,k\}^{\nats}$ is then defined by $s_n=j
\Leftrightarrow n\in N^j$. In general, a $k$-self-shuffling word
defines many different steering words, i.e., the shuffle is not
unique.
We begin with a few simple examples:\\

\noindent{\bf Fibonacci word:}
The Fibonacci infinite word
\[
x=0100101001001010010100\cdots
\]
is defined as the fixed point of  the substitution $\varphi$ given
by $0\mapsto 01$, $1\mapsto 0$. It is readily verified that
$\varphi^2(a)=\varphi(a)a$ for each $a\in \{0,1\}$. Whence,
writing $x=x_0x_1x_2\cdots $ with each $x_i\in \{0,1\}$ we obtain
\begin{align*}
x= x_0x_1x_2\cdots &=\varphi(x_0)\varphi(x_1)\varphi(x_2)\cdots
=\varphi^2(x_0)\varphi^2(x_1)\varphi^2(x_2)\cdots = \\
&\varphi(x_0) x_0\varphi(x_1) x_1\varphi(x_2) x_2\cdots
\end{align*}
which shows that $x$ is self-shuffling.
It is readily verified that the corresponding steering word
$s= 00101001001010010100\cdots$ is equal to the second shift of $x$. \\

\noindent {\bf Period doubling word:} The period-doubling word
\[
x=01000101010001000100010101\cdots
\]
is defined as the fixed point of the substitution $\sigma$ given by $0\mapsto 01$, $1\mapsto 00$.
The period doubling word is also an example of a Toeplitz word (see \cite{CaK}).
It is readily verified that $x$ admits factorizations
$x=\prod_{i=0}^\infty U_iV_i
=\prod_{i=0}^\infty U_i
=\prod_{i=0}^\infty V_i$
with $U_0=0100$, $V_0=01$, and $U_i=\sigma^{i+1}(1)$, $V_i=\sigma^{i}(1)$ for $i\geq1$.
(It suffices to show that each of the above products is fixed by $\sigma)$.
Thus the period-doubling word is self-shuffling.\\


The self-shuffling property appears largely independent of the usual subword complexity of an infinite word
as measured by the number of factors of each given length.
In fact, on one extreme are the purely periodic infinite words all of which are easily seen to be self-shuffling.
The following example illustrates the existence of a self-shuffling word having full complexity:\\

\noindent{\bf A recurrent binary self-shuffling word with full complexity:}
For each positive integer $n$, let $z_n$ denote the concatenation of all words of length $n$
in increasing lexicographic order.
For example, $z_2=00 01 10 11$.
For $i\geq 0$ put
\[
v_i=
\begin{cases}
z_n,\mbox{ if } i=n2^{n-1} \mbox{ for some } n, \\
0^i1^i, \mbox{ otherwise},
\end{cases}
\]
and define
\begin{equation*}
x=\prod_{i=0}^\infty X_i=010100110^3011^30^4010^21^2011^4\cdots,
\label{full}
\end{equation*}
where $X_0=X_1=01$, $X_2=0011$, and for $i\geq 3$, $X_i=0^i y_{i-2} 1^i$,
where  $y_{i-2}=y_{i-3} v_{i-2} y_{i-3}$, and $y_0=\varepsilon$.
We note that $x$ is  recurrent (i.e., each prefix occurs twice)
and has full complexity (since it contains $z_n$ as a factor for every $n)$.

To show that the word $x$ is self-shuffling, we first show that
$X_{i+1}\in \sh (X_i, X_i)$. Take $N_i=\{0, \ldots, i-1, i+1,
\ldots, 2^i-i, 2^i-i+v_{i-1}|_1, 2^{i+1}-i-1\}$, where $u|_1$
denotes the positions $j$ of a word $u$ in which the $j$-th letter
$u_j$ of $u$ is equal to $1$. Then it is not hard to to see that
$X_i=X_{i+1}[N_i]=X_{i+1}[\{1, \ldots, 2^{i+1}\} \backslash N_i]$.
The self-shuffle of $x$ is built in a natural way concatenating
shuffles of $X_i$ starting with $U_0=V_0=01$, so that $X_0 \cdots
X_{i+1}\in \sh (X_0 \cdots X_i, X_0 \cdots X_i)$.\\

The property of being self-shuffling is quite strong.
Nevertheless, every infinite word $s=s_0s_1s_2\cdots \in
\{1,2,\ldots ,k\}^{\nats}$ is a steering word for some
$k$-self-shuffling word. To see this, we define $\ell \colon
\nats\to \nats$ by $\ell (n)=|s_0s_1\cdots s_n|_{s_n}-1$. Let
$\sim$ denote the equivalence relation on $\nats$ generated by
$n\sim \ell(n)$. Then $\sim$ partitions $\nats$ into $r$-many
equivalence classes where $r$ is defined by the condition that
$a^rb$ is a prefix of $s$ for distinct symbols
$a,b\in\{1,2,\ldots,k\}$.
Let $x(s)=x_0x_1x_2\cdots$ be the infinite word over the alphabet
$\{a_1,\ldots,a_r\}$ defined by: for all $n\in\mathbb{N}$ and
$i\in\{1,\ldots,r\}$, $x_n=a_i$ if and only if $n\sim i$. For each
$j\in\{1,\ldots,k\}$, let $t^j\colon\mathbb{N}\to\mathbb{N}$ be
defined by: for all $n\in\mathbb{N}$, $\ell(t^j(n))=n$ and
$s_{t^j(n)}=j$. This defines a new partition of $\mathbb{N}$ into
$k$ classes $N_1,\ldots, N_k$: for each $j\in\{1,\ldots,k\}$,
$N_j=\{t^j(0)<t^j(1)<\cdots\}$. Then, for all $n\in\mathbb{N}$ and
each $i\in\{1,\ldots,r\}$, we have
\[
x_{t^j(n)}=a_i
\Leftrightarrow t^j(n)\sim i
\Leftrightarrow \ell(t^j(n))\sim i
\Leftrightarrow n\sim i
\Leftrightarrow x_n\sim i.
\]
Hence for each $j\in\{1,\ldots,k\}$, $x(s)[N_j]=x(s)$.
This shows that $x(s)$ is $k$-self shuffling and that $s$ steers the described shuffle of $x(s)$.

We illustrate this with an example: Suppose $s\in \{1,2,3\}^\nats$
begins in $s=1111231223123\cdots$ then $(\ell(n))_{n\geq 0}$
begins in $0,1,2,3,0,0,4, 1,2,1,5,3,2,\ldots$ This defines an
equivalence relation with $4$ classes given by :
$\{0,4,5,6,10,\ldots\}$, $ \{1,7,9,\ldots\}$, $\{2,8,12,\ldots\}$
and $\{3,11, \ldots\}$ which in turn defines the
$3$-self-shuffling word $x(s)=abcdaaabcbadc\cdots
\{a,b,c,d\}^\nats$ having $s$ as a steering word.
It follows from Proposition~\ref{app1} that any coding of $x$ is also $3$-self-shuffling.
It follows that every binary word $s\in \{1,2\}^\nats$ beginning in $a^2b$ with $\{a,b\}=\{1,2\}$
determines as above a unique binary self-shuffling word $x$.
In general there is no evident relationship between the words $s$ and $x$.
For instance, $s$ may be uniformly recurrent while $x$ need not be.

The next two propositions show the invariance of self-shuffling words with respect to the action of a morphism:

\begin{proposition}\label{app1}
Let $\A$ and $\B$ be finite non-empty sets and $\tau: \A\to \B^*$ a morphism.
If $x\in \A^\nats$ is $k$-self-shuffling, then so is $\tau(x)\in \B^\nats$.
\end{proposition}

\begin{proof}
Let $x\in \A^\nats$ be $k$-self shuffling. This implies the existence of factorizations
\[
x=\prod_{i=0}^\infty U_i^{(1)}\cdots U_i^{(k)}
    =\prod_{i=0}^\infty U_i^{(1)}
    =\cdots
    =\prod_{i=0}^\infty U_i^{(k)}.
\]

\noindent Whence
\[
\tau(x)=\prod_{i=0}^\infty \tau(U_i^{(1)}\cdots U_i^{(k)})
    =\prod_{i=0}^\infty \tau(U_i^{(1)})\cdots \tau(U_i^{(k)})
    =\prod_{i=0}^\infty \tau(U_i^{(1)})
    =\cdots
    =\prod_{i=0}^\infty \tau(U_i^{(k)})
\]
as required.
\end{proof}

For instance, consider the fixed point $x$ of the substitution
$\psi: a\mapsto abb, $ $b\mapsto a$. It is readily checked that
$x$ is the morphic image of the period doubling word under the
morphism $\tau: 0\mapsto a$, $1\mapsto bb$. This follows from the
fact that $\tau(\sigma^i(0))=\psi^i(a)$,
$\tau(\sigma^i(1))=\psi^i(bb)$ (here $\sigma: 0\mapsto 01,$
$1\mapsto 00$ is the morphism fixing the period doubling word),
which is proved by induction. Since the period doubling word is
$k$-self-shuffling for each $k\geq 2$, the same is true of $x$.

Following \cite{Du}, if $x\in \A^\nats$ is uniformly recurrent,
and $u$ a non-empty prefix of $x$, then one defines the derived
sequence $\mathcal{D}_u(x)$ by coding $x$ as a concatenation of
first returns to $u$. Recall that a nonempty word $v$ is called a
\emph{first return} to a factor $u$ of $x$ if $vu$ is a factor of
$x$, the word $u$ is a prefix of $vu$ and $vu$ does not contain
other occurrences of $u$ than
suffix and prefix. 
The following is an immediate consequence of
Proposition~\ref{app1} since in fact $x$ is the morphic image of
$\mathcal{D}_u(x)$:

\begin{corollary}\label{derived}
Let $x\in \A^\nats$ be a uniformly recurrent word and $u$ a non-empty prefix of $x$.
If the derived word $\mathcal{D}_u(x)$ is $k$-self-shuffling, then $x$ is $k$-self-shuffling.
\end{corollary}

The notation $w=v^{-r}u$ means $u=v^r w$.

\begin{proposition}\label{app2}
Let $\tau: \A\to \A^*$ be a morphism,
and $x\in \A^\nats$ be a fixed point of $\tau$.
\begin{enumerate}
    \item Let $u$ be a prefix of $x$ and $l$ be a positive integer
        such that $\tau^l(a)$ begins in $u$ for each $a\in \A$.
        Then if $x$ is $k$-self-shuffling, then so is $u^{-1}x$.
    \item Let $u\in \A^*$,  and let $l$ be a positive integer such that $\tau^l(a)$
        ends in $u$ for each $a\in \A$.
Then if $x$ is $k$-self-shuffling, then so is $ux$.
\end{enumerate}
\end{proposition}

\begin{proof}
We prove only item (1) since the proof of (2) is essentially identical.
Suppose
\[
x=\prod_{i=0}^\infty U_i^{(1)}\cdots U_i^{(k)}
=\prod_{i=0}^\infty U_i^{(1)}
=\cdots
=\prod_{i=0}^\infty U_i^{(k)}
\]
with each $U_i^{(j)}\in \A^+$.
By assumption, for each $i\geq 1$ and $j= 1,2,\ldots ,k$,
we can write $\tau^l(U_i^{(j)})=uV_i^{(j)}$ with each $V_i^{(j)}\in \A^*$.
Put $X_i^{(j)}=V_i^{(j)}u$.
Then since
\[
x=\tau^l(x)
=\prod_{i=0}^\infty \tau^l(U_i^{(1)}\cdots U_i^{(k)})
=\prod_{i=0}^\infty \tau^l(U_i^{(1)})\cdots \tau^l(U_i^{(k)})
=\prod_{i=0}^\infty \tau^l(U_i^{(1)})
=\cdots
= \prod_{i=0}^\infty \tau^l(U_i^{(k)}),
\]
we deduce that
\[
u^{-1}x= \prod_{i=0}^\infty X_i^{(1)}\cdots X_i^{(k)}
=\prod_{i=0}^\infty X_i^{(1)}
=\cdots
=\prod_{i=0}^\infty X_i^{(k)}.
\]
\end{proof}

A substitution $\tau: \A\to \A^*$ is \emph{primitive} if for each
pair $a,b \in \A$ there exists $n$ such that $b$ occurs in
$\tau^n(a)$. 

\begin{corollary}\label{primitive}
Let $\tau: \A\to \A^*$ be a primitive substitution, and $a\in \A$.
Suppose $\tau(b)$ begins (respectively ends) in $a$  for each letter $b\in \A$.
Suppose further that the fixed point $\tau^\infty(a)$ is $k$-self-shuffling.
Then every right shift (respectively left shift)  of $\tau^\infty(a)$ is $k$-self-shuffling.
\end{corollary}

\begin{proof} Since $\tau$ is primitive, the lengths of the images
$\tau^l(b)$, $b\in \A$, grow as $l$ grows. Suppose $\tau(b)$
begins (respectively ends) with $a$  for each letter $b\in \A$
(the other case is symmetric). So for every prefix $u$ of
$\tau^\infty(a)$ there exists $l$ such that $u$ is a prefix of
$\tau^l(b)$ for all letters $b$. Since the word $\tau^\infty(a)$
is also a fixed point of the morphism $\tau^l$, the claim follows
from Proposition~\ref{app2}. \end{proof}

\begin{remark}
Since the Fibonacci word is self-shuffling
and is fixed  by the primitive substitution $0\mapsto 01$, $1\mapsto 0$,
it follows from Corollary~\ref{primitive} that  every tail of the Fibonacci word is self-shuffling.
\end{remark}

There are a number of necessary conditions that a self-shuffling word
must satisfy, which may be used to deduce that a given word is not self-shuffling.
For instance:

\begin{definition}
The {\em shuffling delay} of a self-shuffling word $x$ is the
length of the shortest prefix $u$ of $x$ such that
$(ua)^{-1}x\in\sh(u^{-1}x,a^{-1}x)$ where $a$ is the letter
following the prefix $u$ in $x$. In other words, the shuffling
delay is the length of the shortest prefix of $x$ after which one
can actually start self-shuffling $x$.
\end{definition}

Two finite words are \emph{abelian equivalent} if for each letter
$a$ they have the same number of occurrences of $a$. A word $u$
has a \emph{border} (resp., an \emph{Abelian border}) of length
$n$, $0<n<|u|$, if the prefix of length $n$ of $u$ is equal
(resp., Abelian equivalent) to its suffix of length $n$. A word
$u$ is \emph{unbordered} (resp., \emph{Abelian unbordered}) if it
does not have borders (resp., Abelian borders).

\begin{proposition}\label{abborders}
If $x\in \A^\nats$ is self-shuffling, then for each positive
integer $N$ there exists a positive integer $M$ such that every
prefix  $u$ of $x$ with $|u|\geq M$ has an Abelian border $v$ with
$|u|/2 \geq |v|\geq N$. Moreover, every prefix of $x$ longer than
the shuffling delay is Abelian bordered. In particular, $x$ must
begin in only a finite number of Abelian unbordered words.
\end{proposition}

\begin{proof}
Suppose to the contrary that there exist factorizations
$x=\prod_{i=0}^\infty U_iV_i=\prod_{i=0}^\infty U_i=\prod_{i=0}^\infty V_i$
with $U_i,V_i \in \A^+$, and there exists $N$ such that for every $M$
there exists a prefix $u$ of $x$ with $|u|\geq M$ which has no Abelian borders of length between $N$ and $|u|/2$.
Take $M=|\prod_{i=0}^{N} U_iV_i|$ and a prefix $u$ satisfying these conditions.
Then there exist non-empty proper prefixes $U'$ and $V'$ of $u$ such that $u\in \sh (U',V')$ with $|U'|,|V'|>N$.
Without loss of generality we may assume $|V'|\le |u|/2$.
 Writing $u=U'U''$ it follows that $U''$ and $V'$ are Abelian equivalent.
This contradicts that $u$ has no Abelian borders of length between $N$ and $|u|/2$.

For the second part of the statement, just observe that every prefix of $x$ longer than the shuffling delay
is the shuffle of two proper prefixes, and hence is Abelian bordered.
\end{proof}

We illustrate some applications of Proposition~\ref{abborders}:\\

\noindent{\bf Fibonacci word (revisited):} We already saw that the
Fibonacci word $x=0100101001\cdots$ is self-shuffling. In
contrast, the word $y=0x$ is not self-shuffling. It is well known
that $y$ begins in infinitely many prefixes of the form $0B1$ with
$B$ a palindrome. It is clear that $0B1$ is  Abelian unbordered.
It follows from Proposition \ref{abborders} that $y$ is not self-shuffling. \\

\noindent{\bf Paper-folding word:}
The paper-folding word
\[
x= 00100110001101100010\cdots
\]
is a Toeplitz word generated by the pattern  $u = 0?1?$ (see,
e.g., \cite{CaK}). It is readily verified that $x$ begins in
arbitrarily long Abelian unbordered words and hence by Proposition
\ref{abborders} is not self-shuffling. More precisely, the
prefixes $u_j$ of $x$ of length $n_j=2^j-1$ are Abelian
unbordered. Indeed, it is verified that for each $k<n_j$, we have
$|{\rm pref}_{k}(u_j)|_0>k/2$ while $|{\rm suff}_k(u_j)|_0\leq
k/2$.
Here ${\rm pref}_{k}(u)$ (resp., ${\rm suff}_k(u)$) denotes the prefix (resp., suffix) of length $k$ of a word $u$. \\

\noindent{\bf A $3$-self-shuffling word which is not self-shuffling:}
Let $y$ denote the fixed point of the substitution $\sigma: 0 \mapsto 0001$ and $1 \mapsto 0101$, and put
\[
x=0^{-2}y=01000100010101000100010001010100010001000101010001010100\cdots
\]
Then for each prefix $u_j$ of $x$ of length $4^j-2$, the longest
Abelian border of $u_j$ of length less than or equal to
$(4^j-2)/2$ has length  $2$. This could be proved by induction on
$j$ using the fact that $\sigma^j(1)$ differs from $\sigma^j(0)$
in $j$ positions and has $1$ instead of $0$ in these positions.
Hence $x$ is not self-shuffling (see Proposition~\ref{abborders}).
The 3-shuffle is given by the following: \small
\begin{align*}
& U_0=0100,& &\hskip-9pt U_1=01,& &\hskip-10pt \dots,\hskip-8pt & &U_{4i+2}=\varepsilon,&
&\hskip-6pt U_{4i+3}=\sigma^{i+1}(0100),\\
& & & & & & &U_{4i+4}=\sigma(0),& & \hskip-6pt U_{4i+5}=(\sigma(0))^{-1}\sigma^{i+1}(01), \\
& V_0=0100,& &\hskip-9pt V_1=01,& &\hskip-10pt \dots, \hskip-8pt & &V_{4i+2}=(\sigma(0))^{-1}\sigma^{i+1}(0),&
& \hskip-6pt V_{4i+3}=\sigma(0),\\
& & & & & & &V_{4i+4}=(\sigma(0))^{-1}\sigma^{i+1}(01)\sigma(0),& & \hskip-6pt V_{4i+5}=\varepsilon,\\
& W_0=01,& &\hskip-9pt W_1=(\sigma(0))^2,& &\hskip-10pt \dots,\hskip-8pt & &
W_{4i+2}=\varepsilon,& & \hskip-6pt W_{4i+3}=(\sigma(0))^{-1}\sigma^{i+1}(01),\\
& & & & & & &W_{4i+4}=\varepsilon,& &\hskip-6pt W_{4i+5}=\sigma^{i+2}(0)\sigma(0).
\end{align*}
\normalsize
It is then verified  that
\[
x=\prod_{i=0}^\infty U_i V_i W_i=\prod_{i=0}^\infty U_i
=\prod_{i=0}^\infty V_i=\prod_{i=0}^\infty W_i.
\] To prove this, this it is enough to prove that the morphism
$\sigma$ fixes all the four products of words, from which it follows that $x$ is $3$-self-shuffling.\\

An extension of the abelian borders argument (Proposition \ref{abborders})
gives both a necessary and sufficient condition
for self-shuffling in terms of Abelian borders (which is however difficult to check in practice).
For $u\in \A^*$ let $\Psi(u)$ denote the \emph{Parikh vector} of $u$, i.e., $\Psi(u)=(|u|_a)_{a\in \A}$.

\begin{definition}\label{def_graph}
For $x\in\A^\mathbb{N}$ and integer $k\geq 2$, define a directed
graph $G_x^k=(V_x^k,E_x^k)$ with the vertex set
\[
V_x^k=\{(i_1,\ldots, i_k)\in\mathbb{N}^k \colon \sum_{j=1}^k
\Psi(\pref_{i_j}(x))=\Psi(\pref_{i_1+\cdots+i_k}(x))\},
\]
and the edge set
\[
E_x^k=\big\{\big((i_1,\ldots, i_k),(i'_1,\ldots, i'_k)\big)\in V_x^k\times V_x^k\colon
i_j\le i'_j\mbox{ for } j=1,\ldots,k
\mbox{ and }\sum_{j=1}^k i_j+1=\sum_{j=1}^k i'_j\big\}.
\]
We say that $G_x^k$ {\it connects $\vec{0}$ to $\vec{\infty}$} if
there exists an infinite path $v^0v^1v^2\cdots$ in $G_x^k$ such
that $v^0=(0,\ldots,0)$ and $v_j^n\to\infty$ as $n\to\infty$ for
any $j=1,\ldots,k$, where $v^n=(v_1^n,\ldots,v_k^n)\in V_x^k$.
\end{definition}

\begin{theorem}\label{graph}
An infinite sequence $x\in\A^\nats$ is $k$-self-shuffling if and
only if the graph $G_x^k$ connects $\vec{0}$ to $\vec{\infty}$.
\end{theorem}

\begin{proof}
If $x=x_0x_1\cdots$ with the $x_i \in \A$ is $k$-self-shuffling,
then there exists a partition of $\nats$ into infinite
subsets $N^j\subset\nats$ for $j=1,\ldots,k$ such that $x[N^j]=x$.
Let $N^j=\{N^j_0<N^j_1<N^j_2<\cdots\}$ for every $j\in\{1,\ldots,k\}$.
For any $j\in\{1,\ldots,k\}$ and $n\in\nats$,
let $t_j(n)$ be the integer $m$ such that $N^j_{m-1}<n\le N^j_m$.
Then $\Psi(\pref_{n}(x)) =\sum_{j=1}^k\Psi(\pref_{t_j(n)}(x[N^j]))$.
Since $x[N^j]=x$ for $ j=1,\ldots,k$,
we have
\[
\sum_{j=1}^k \Psi(\pref_{t_j(n)}(x))=\Psi(\pref_{n}(x)).
\]
Moreover $t_1(n)+\cdots+t_k(n)=n$.
Therefore $v^n:=(t_1(n),\ldots,t_k(n))\in V^k_x$ for all $n\in\nats$.
Since
\[
t_j(n+1)= \left\{
\begin{array}{ll}
t_j(n)+1&\mbox{ if } n\in N^j\\
t_j(n)&\mbox{ else, }
\end{array}
\right.
\]
we obtain $(v^n,v^{n+1})\in E^k_x$.
Since this holds for any $n\in\nats$ and each $N^j$ is an infinite set,
$G_x^k$ connects $\vec{0}$ to $\vec{\infty}$.

Conversely, assume that $G_x^k$ connects $\vec{0}$ to
$\vec{\infty}$. Let $v^0v^1v^2\cdots$ be an infinite path in
$G_x^k$ such that $v^0=(0,\ldots,0)$ and $v_j^n\to\infty$ as
$n\to\infty$ for any $j=1,\ldots,k$, where
$v^n=(v_1^n,\ldots,v_k^n)\in V_x^k$. Define
\[
N^j=\{n\in\nats\colon v_j^{n+1}>v_j^n\}\mbox{ for } j=1,\ldots,k.
\]
Then the sets $N^j$ give a partition of $\nats$.
Let $N^j=\{N^j_0<N^j_1<N^j_2<\cdots\}$ for $j=1,\ldots,k$.
For any $j=1,\ldots,k$ and $m\in\nats$, let $n=N^j_m$.
Since $(v^n,v^{n+1})\in E_x^k$, we have
\[
\Psi(\pref_{n+1}(x))-\Psi(\pref_{n}(x))
=\Psi(\pref_{m+1}(x))-\Psi(\pref_{m}(x)).
\]
Hence $(x[N^j])_m=x_n=x_m$.
Thus, the sets $N^j$ satisfy Definition~\ref{k-self-shuffling}, so we have a $k$-self-shuffle.
\end{proof}

\section{The Thue-Morse word is self-shuffling}

Theorem~\ref{graph} gives a constructive necessary and
sufficient condition for self-shuffling since a path to infinity defines a self-shuffle.
While the sufficient conditions in the previous section can be applied to show that certain words are not self-shuffling,
in general to prove that a word is self-shuffling, one must actually explicitly exhibit a shuffle.

\begin{theorem}\label{th_TM}
The Thue-Morse word $\mathbf{T}=011010011001\cdots$ fixed by the
substitution $\tau $ mapping $0\mapsto 01$ and $1\mapsto 10$ is
self-shuffling.
\end{theorem}

\begin{proof}
For $u\in\{0,1\}^*$ we denote by $\bar{u}$ the word obtained from $u$ by exchanging $0$s and $1$s.
Let $\sigma : \{1,2,3,4\}\to \{1,2,3,4\}^*$ be the substitution defined by
\[
\sigma (1)=12,\quad \sigma (2)=31,\quad \sigma(3)=34, \quad  \sigma (4)=13.
\]
Set  $u=01101$ and $v=001;$ note that $uv$ is a prefix of $\mathbf{T}$.
Also define morphisms $g,h:  \{1,2,3,4\}\to \{0,1\}^*$ by
\[
g (1)=v\bar{u},\quad g (2)=\bar{v}\bar{u},\quad g (3)=\bar{v}u, \quad  g (4)=vu
\]
and
\[
h (1)=uv,\quad h (2)=\bar{u}\bar{v},\quad h (3)=\bar{u}\bar{v},
\quad  h (4)=uv.
\]

We will make use of the following lemmas:

\begin{lemma}\label{shuff1}
$g(\sigma (a))\in \sh(g(a),h(a))$ for each $a\in \{1,2,3,4\}$.
In particular $ug(\sigma(1))\in \sh(ug(1),h(1))$.
\end{lemma}

\begin{proof}
For $a=1$ we note that
\[
g(\sigma(1))=g(12)=v\bar{u}\bar{v}\bar{u}=0011001011010010.
\]
Factoring $0011001011010010=0\cdot011\cdot0\cdot010\cdot11\cdot01\cdot0010$ we obtain
\[
g(\sigma(1))\in \sh(00110010,01101001)=\sh (v\bar{u},uv)=\sh(g(1),h(1)).
\]
Similarly, for $a=2$ we have
\[
g(\sigma(2))=g(31)=\bar{v}uv\bar{u}=1100110100110010.
\]
Factoring $1100110100110010=1\cdot100\cdot1\cdot1\cdot010\cdot0110\cdot010$ we obtain
\[
g(\sigma(2))\in \sh(11010010,10010110)=\sh(\bar{v}\bar{u},\bar{u}\bar{v})=\sh (g(2),h(2)).
\]

\noindent Exchanging  $0$s and $1$s in the previous two shuffles yields
\[
g(\sigma(3))=g(34)=\bar{v}uvu\in \sh(\bar{v}u,\bar{u}\bar{v})=\sh(g(3),h(3))
\]
and
\[
g(\sigma(4))=g(13)=v\bar{u}\bar{v}u\in \sh(vu,uv)=\sh(g(4),h(4)).
\]
\end{proof}

It is readily verified that

\begin{lemma}
$h(\sigma (a))= \tau(h(a))$ for each $a\in \{1,2,3,4\}$.
\end{lemma}

\noindent Let $w=w_0w_1w_2w_3\cdots$ with $w_i\in \{1,2,3,4\}$
denote the fixed point of $\sigma$ beginning in $1$.
As a  consequence of the previous lemma we deduce that 

\begin{lemma}\label{T1}
$\mathbf{T}=h(w)$.
\end{lemma}

\begin{proof}
In fact $\tau(h(w))=h(\sigma(w))=h(w)$ from which it follows that $h(w)$ is one of the two fixed points of $\tau$.
Since $h(w)$ begins in $h(1)$ which in turn begins in $0$, it follows that  $\mathbf{T}=h(w)$.
\end{proof}

\begin{lemma}\label{T2}
$\mathbf{T}=ug(w)$.
\end{lemma}

\begin{proof}
It is readily verified that:
\[
ug(1)=h(1)\bar{u}, \quad\bar{u}g(2)=h(2)\bar{u}, \quad
\bar{u}g(3)=h(3)u, \quad ug(4)=h(4)u.
\]
Moreover, each occurrence of $g(1)$ and $g(4)$ in $ug(w)$ is preceded by $u$
while each occurrence of $g(2)$ and $g(3)$ in $ug(w)$ is preceded by $\bar{u}$.
It follows that $ug(w)=h(w)$ which by the preceding lemma equals $\mathbf{T}$.
\end{proof}

\noindent To proceed with the proof of Theorem \ref{th_TM}, set
\[
A_0=ug(\sigma(w_0))\quad \mbox{ and }\quad A_i=g(\sigma(w_i)), \mbox{ for }  i\geq 1
\]
\[
B_0=ug(w_0)\quad \mbox{ and }\quad B_i=g(w_i), \mbox{ for } i\geq
1
\]
\[
C_i=h(w_i), \mbox{ for }  i\geq 0.
\]
It follows from Lemma~\ref{T1} and Lemma~\ref{T2} that $
\mathbf{T}=\prod_{i=0}^\infty A_i=\prod_{i=0}^\infty
B_i=\prod_{i=0}^\infty C_i $ and it follows from
Lemma~\ref{shuff1}  that $A_i\in \sh(B_i,C_i)$ for each $i\geq 0$.
Hence $\mathbf{T}\in \sh(\mathbf{T},\mathbf{T})$ as required.
Theorem \ref{th_TM} is proved.
\end{proof}

\section{Infinite Lyndon words are not self-shuffling}

\begin{definition}\label{Lyn}
Let $y\in \A^\nats$.
We say $y$ is Lyndon if there exists an order $\leq$ on $\A$
with respect to which $y$ is lexicographically smaller than all its proper suffixes (or tails).
\end{definition}

Recall that if $(\A,\leq )$ is a  linearly ordered set, then
$\leq$ induces the {\it lexicographic order}, denoted $\llex$, on
$\A^+$ and $\A^\nats$ defined as follows: If $u,v \in \A^+$ (or
$\A^\nats)$ we write $u\llex v$ if either $u=v$ or if  $u$ is
lexicographically smaller than $v$.  In the latter case we write
$u\sllex v$. Thus  $y\in \A^\nats $ is Lyndon if and only if there
exists an order $\leq$ on $\A$ with respect to which $y\sllex y'$
for all proper suffixes $y'$ of $y$. Note that the property of
being Lyndon is an intrinsic property of a word in the sense that
if $y$ is Lyndon and $z$ is word isomorphic to $y$, then $z$ is
Lyndon. For instance $10^\omega$ and $01^\omega$, are each Lyndon.
Similarly, any suffix of $x=101001000100001\cdots$ beginning in
$1$ is Lyndon. Note that in this last example, all Lyndon words in
the shift orbit closure of $x$ begin in $1$. Indeed, if $y$ is any
Lyndon word beginning with $0$ then the blocks if $0$'s are
bounded in length. It follows from Definition~\ref{Lyn} that a
Lyndon word is never purely periodic (although it may be
ultimately periodic) as in the first two examples above. We also
note that if $y$ is Lyndon with respect to $\leq$, then every
prefix $u$ of $y$ is {\it minimal} in $y$ with respect to $\leq$,
i.e., $u\llex v$ for all factors $v$ of $y$ with $|v|=|u|$.

\begin{theorem}\label{extremal}
Let $y\in \A^\nats$ be Lyndon relative to some order $\leq$ on $\A$
and let $z\in \A^\nats$ be any point in the shift orbit closure of $y$.
Then for each $w\in \sh(y,z)$,  we have $w \sllex z$.
In particular, taking $z=y$, it follows that $y$ is not self-shuffling.
\end{theorem}




In order to prove Theorem~\ref{extremal} we will make use of the following  four lemmas.
In each of the following lemmas, assume $(\A,\leq )$ is a  linearly ordered set and $x\in \A^\nats$.

\begin{lemma}\label{prodmin}
Let $u=u_1u_2\cdots u_n$ be a factor of $x$ with each $u_i$ minimal in $x$.
Then $u$ is minimal in $x$.
\end{lemma}

\begin{proof}
It is readily verified by induction on $k$ that each prefix $u_1u_2\cdots u_k$ is minimal in $x$.
\end{proof}

\begin{lemma}\label{lup}
Let $u$ be a minimal factor of $x$ and let $v$ be the longest unbordered prefix of $u$.
Then $v$ is a period of $u$, i.e., $u$ is a prefix of $v^n$ for some positive integer $n$.
\end{lemma}

\begin{proof}
The result of the lemma is clear in case $v=u$.
So suppose $|v|<|u|$ and let us write $u=vu'$.
Since every prefix of $u$ of length longer than $|v|$ is bordered, and since $v$ itself is unbordered,
it follows that $u'$ is a product $v_kv_{k-1}\cdots v_1$ where each $v_i$ is a prefix of $u$.
In fact, since $u$ is bordered, there exists a border $v_1$ of $u$.
Moreover since $v$ is unbordered, the suffix $v_1$ of $u$ does not overlap $v$,
and hence $v_1$ is a suffix of $u'$.
If $v_1=u'$ we are done, otherwise the prefix $uv_1^{-1}$ of $u$ admits a border $v_2$.
Again since $v$ is unbordered, the suffix $v_2$ of $uv_1^{-1}$ does not overlap $v$ and hence $v_2v_1$ is a suffix of $u'$.
(See Figure~\ref{del}).
\begin{figure}[ht]
  \begin{center}
\begin{picture}(90, 20)(0,10)
\gasset{Nw=.3,Nh=2,Nmr=0,fillgray=0} 
\gasset{ExtNL=y,NLdist=1,NLangle=-90} 

   \node(X0)(0,15){$$}
   \node(X1)(40,15){$$}

   \node(X4)(90,15){$$}

  \gasset{Nw=0,Nh=0,Nmr=0,fillgray=0}

   \node(X11)(7,15){$$}
    \node(X12)(33,15){$$}
    \node(X13)(10,15){$$}
   \node(X2)(50,15){$$}
  \node(X3)(83,15){$$}

  \node(x)(-10,17){$u=$}

 \drawedge[AHnb=0,ELside=r,  linewidth=0.5](X0,X1){$v$}
   \drawedge[AHnb=0,ELside=r](X1,X4){$u'$}
   \drawedge[AHnb=0,ELside=l](X2,X3){}
    \drawedge[AHnb=0,ELside=l](X3,X4){}

\gasset{curvedepth=5}

          \drawedge[AHnb=0,ELside=l](X2,X3){$v_2$}

          \gasset{curvedepth=5}
          \drawedge[AHnb=0,ELside=l](X0,X12){$v_2$}

          \gasset{curvedepth=3}
          \drawedge[AHnb=0,ELside=l](X1,X2){$v_3$}

          \gasset{curvedepth=3}

\drawedge[AHnb=0,ELside=l](X3,X4){$v_1$}

\gasset{curvedepth=-3}
\drawedge[AHnb=0,ELside=r](X0,X11){$v_1$}
\end{picture}
  \end{center}
  \caption{$u'$ as a product of prefixes of $u$.}
  \label{del}
\end{figure}
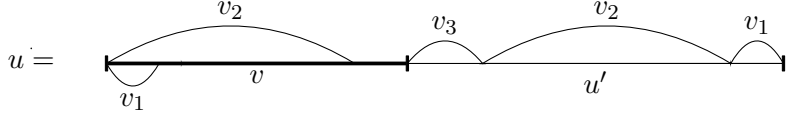
Continuing in this way we can write $u'$ as a product of prefixes of $u$
whence $u'$ is a product of minimal factors of $x$. It follows from Lemma~\ref{prodmin}
that $u'$ is a minimal factor of $x$ and hence $u'$ is  both a prefix and a suffix of $u$.
If $|u'|\leq |v|$ then $u'$ is a prefix of $v$ and hence $u$ is a prefix of $v^2$.
Otherwise, if $|u'|>|v|$ then the occurrences of $u'$ at the beginning and end of $u$ overlap;
whence $v$ is a period of $u$ (see Figure~\ref{del2}).
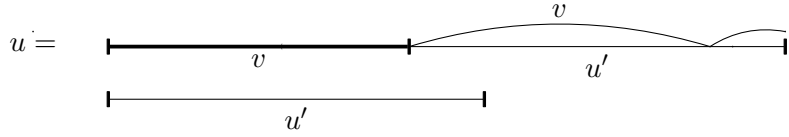
\begin{figure}[ht]
  \begin{center}
\begin{picture}(85,20)(0,5)
\gasset{Nw=.3,Nh=2,Nmr=0,fillgray=0} 
\gasset{ExtNL=y,NLdist=1,NLangle=-90} 

   \node(X0)(0,15){$$}
   \node(X1)(40,15){$$}

   \node(X4)(90,15){$$}

   \node(Y0)(0, 8){$$}
   \node(Y1)(50,8){$$}

  \gasset{Nw=0,Nh=0,Nmr=0,fillgray=0}

   \node(X11)(7,15){$$}
    \node(X12)(23,15){$$}
   \node(X2)(80,15){$$}
  \node(X3)(83,15){$$}

  \node(x)(-10,17){$u=$}

    \node(X3)(90,17){$$}

  \drawedge[AHnb=0,ELside=r,  linewidth=0.5](X0,X1){$v$}
   \drawedge[AHnb=0,ELside=r](X1,X4){$u'$}

    \drawedge[AHnb=0,ELside=r](Y0,Y1){$u'$}

\gasset{curvedepth=3}

          \drawedge[AHnb=0,ELside=l](X1,X2){$v$}
          \gasset{curvedepth=1}
          \drawedge[AHnb=0,ELside=l](X2,X3){}
\end{picture}
  \end{center}
  \caption{$u$ has period $v$.}
  \label{del2}
\end{figure}
\end{proof}

\begin{lemma}\label{period}
Let $u$ and $v$ be factors of $x$ with $u$ minimal.
Then either $uv \sllex v$ or else $v$ is minimal.
\end{lemma}

\begin{proof}
Suppose $\neg (uv\sllex v)$. Since $u$ is minimal, $v$ is a prefix of $uv$.
If $|v|\leq |u|$ then $v$ is a prefix of $u$ and hence $v$ is minimal.
If $|v|>|u|$ then the prefix $v$ of $uv$ overlaps the suffix $v$ of $uv$
whence $u$ is a period of $uv$  and hence $u$ is also a period of $v$.
Thus by Lemma~\ref{prodmin} we deduce that $v$ is minimal.
\end{proof}

Given two finite non-empty words $u$ and $v$ we say that $s\in \sh(u,v)$
is a {\it proper shuffle} of $u$ and $v$ if there exists a positive integer $k$ and factorings
\[
u=\prod _{i=0}^k  U_i\quad \mbox{ and } \quad v=\prod _{i=0}^k  V_i
\]
and  either \[s=\prod _{i=0}^k  U_iV_i\] with each of $U_0,V_0$ and $U_1$ non-empty, or
\[
s=\prod _{i=0}^k  V_iU_i
\]
with each of $V_0,U_0$ and $V_1$ non-empty.
In other words, $s$ is a proper shuffle of $u$ and $v$ means that $s$ is obtained as a non-trivial shuffle of $u$ and $v$.
We let $\sh^*(u,v)$ denote the set of all proper shuffles of $u$ and $v$.

Let ${\mathcal{C}} (x)$ denote the set of all factors $v$ of $x$
with the property that no suffix of $v$ (including $v$ itself) is minimal in $x$.

\begin{lemma}\label{propmin}
Let $u$ and $v$ be factors of $x$ with $u$ minimal in $x$ and $v\in {\mathcal{C}}(x)$.
Let $s\in \sh^*(u,v)$. Then $s\sllex v$.
\end{lemma}

\begin{proof} 
Let $0$ denote the minimal element of $\A$ with respect to the linear order $\leq$.
We proceed by induction on $|u|+|v|$. Since $s$ is assumed to be a proper shuffle of $u$ and $v$,
it follows that  $|u|+|v|\geq 3$.
For the base case $|u|+|v|=3$ we have either $u=0a$ and $v=b$ with $b\neq 0$ or $u=0$ and $v=cd$ with $d\neq 0$.
In the first case $s=0ba\sllex b$ and in the second case $s=c0d \sllex cd$.

Let $N\geq 4$. By induction hypothesis we suppose the result of the proposition is true whenever $|u|+|v|<N$.
Now suppose $u$ and $v$ are factors of $x$ with $u$ minimal in $x$,  $v\in \mathcal{C}(x)$ and $|u|+|v|=N$.
Let
\[
u=\prod _{i=0}^k  U_i\quad \mbox{ and } \quad v=\prod _{i=0}^k V_i
\]
be factorings of $u$ and $v$ with each $U_i$ and $V_i$ non-empty except for possibly  $U_k$ or $V_k$. \\

\noindent {\bf Case 1.} We consider first  the case in which $v$
dishes out the initial segment $V_0$ of $s$, i.e., $ s=\prod
_{i=0}^k  V_iU_i. $ Set $v'=V_1\cdots V_k$ so that $v=V_0v'$ (see
Figure~\ref{del3}).
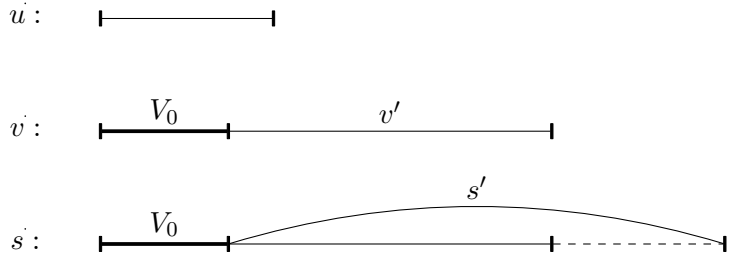
\begin{figure}[htbp]
  \begin{center}
\begin{picture}(90,40)(0,8)
\gasset{Nw=.3,Nh=2,Nmr=0,fillgray=0} 
\gasset{ExtNL=y,NLdist=1,NLangle=-90} 

 \node(Y0)(0,40){$$}
 \node(Y1)(23,40){$$}
   \node(W0)(0,10){$$}
   \node(W1)(17,10){$$}
   \node(W2)(60,10){$$}
   \node(W3)(83,10){$$}
\node(Z0)(0,  25){$$}
\node(Z1)(17,25){$$}
\node(Z2)(60,25){$$}

  \gasset{Nw=0,Nh=0,Nmr=0,fillgray=0}

  \node(w)(-10,12){$s:$}
  \node(y)(-10,42){$u:$}
  \node(z)(-10,27){$v:$}

  \drawedge[AHnb=0,ELside=l,linewidth=0.5](W0,W1){$V_0$}
   \drawedge[AHnb=0,ELside=l](W1,W2){}
   \drawedge[AHnb=0,ELside=l,dash={1.}](W2,W3){}
  \drawedge[AHnb=0,ELside=l,linewidth=0.5](Z0,Z1){$V_0$}
  \drawedge[AHnb=0,ELside=l](Z1,Z2){$v'$}
  \drawedge[AHnb=0,ELside=l](Y0,Y1){}

  \gasset{curvedepth=5}

          \drawedge[AHnb=0,ELside=l](W1,W3){$s'$}
\end{picture}
  \end{center}
  \caption{Case 1: $v$ dishes out the  initial segment of $s$.}
  \label{del3}
\end{figure}
Since $s$ is assumed to be a proper shuffle of $u$ and $v$, it
follows that $v'$ is non-empty and $v' \in \mathcal{C}(x)$.
Let us write $s=V_0s'$.
We have two possibilities: either $s'=uv'$ or $s'\in \sh^* (u,v')$.
 If $s'=uv'$, then, because $v'$ is not minimal, we have $s'=uv'\sllex v'$ by Lemma~\ref{period}.
 If $s'\in \sh^* (u,v')$, then by induction hypothesis  $s'\sllex v'$.
 Then, in both cases, we have $s=V_0s'\sllex V_0v'=v$. \\

\noindent{\bf Case 2.} We may now suppose that $u$ dishes out the
initial segment $U_0$ of $s$, i.e., $ s=\prod _{i=0}^k  U_iV_i. $
Let $Ub$ denote the shortest non-minimal prefix of $v$ with $U\in
\mathbb{A}^*$ and $b\in  \mathbb{A}$. Then $Ub$ belongs to
$\mathcal{C}(x)$ for otherwise $Ub$ would be the concatenation of
two minimal factors, and hence itself minimal by
Lemma~\ref{prodmin}. Since $U_0$ is a prefix of $s$ and $U_0$ is
minimal, we deduce that either $U_0$ is a prefix of $v$ (and hence
of $U)$ or $s\sllex v$. Thus we may suppose that $U_0$ is a prefix
of $U$ so that $|U|\geq |U_0|$.
We consider two sub-cases according to the length of $u$. \\

\noindent{\bf Case 2.1.}
Suppose $|u|>|U|$, i.e., $U$ is a proper prefix of $u$.
In this case let $s'$ denote the prefix of $s$ of length $|Ub|$ (see Figure~\ref{del 6}).
\begin{figure}[htbp]
  \begin{center}
\begin{picture}(90,45)(0,5)
\gasset{Nw=.3,Nh=2,Nmr=0,fillgray=0} 
\gasset{ExtNL=y,NLdist=1,NLangle=-90} 

 \node(Y0)(0, 40){$$}
 \node(Y1)(20, 40){$$}
 \node(Y2)(60,40){$$}
   \node(W0)(0,10){$$}
   \node(W1)(20,10){$$}
   \node(W12)(28,10){$$}
   \node(W2)(60,10){$$}
\node(Z12)(8,25){$$}
\node(Z0)(0,25){$$}
\node(Z2)(60,25){$$}

\gasset{Nw=1,Nh=1,Nmr=1,fillgray=0} 

\node(Z3)(62,25){$b$}
\node(W3)(62,10){$$}

  \gasset{Nw=0,Nh=0,Nmr=0,fillgray=0}
  \node(Z1)(20,25){$$}
  \node(w)(-10,12){$s:$}
  \node(y)(-10,42){$u:$}
  \node(z)(-10,27){$v:$}

   \node(W4)(100,10){$$}
  \node(Y3)(74,40){$$}

  \drawedge[AHnb=0,ELside=r, linewidth=0.5](W0,W1){$U_0$}
  \drawedge[AHnb=0,ELside=r, linewidth=0.5](W1,W12){$V_0$}
  \drawedge[AHnb=0,ELside=r, linewidth=0.5](Z0,Z12){$V_0$}
   \drawedge[AHnb=0,ELside=l](W1,W2){}
   \drawedge[AHnb=0,ELside=l,dash={1.}](W3,W4){}
   \drawedge[AHnb=0,ELside=l,dash={1.}](Y2,Y3){}

    \node(Z4)(74,25){$$}
    \drawedge[AHnb=0,ELside=l,dash={1.}](Z3,Z4){}

  \drawedge[AHnb=0,ELside=r](Z0,Z1){}
  \drawedge[AHnb=0,ELside=r, linewidth=0.5](Y0,Y1){$U_0$}
  \drawedge[AHnb=0,ELside=l](Y1,Y2){}
  \drawedge[AHnb=0,ELside=l](Y0,Y2){$U$}
   \drawedge[AHnb=0,ELside=l](Z0,Z2){$U$}

 \drawedge[AHnb=0,ELside=l](Z0,Z2){}

\drawedge[AHnb=0,ELside=r](W1,W12){}

          \drawedge[AHnb=0,ELside=l](W1,W2){}

\gasset{curvedepth=5}
\drawedge[AHnb=0,ELside=l](W0,W3){$s'$}
\end{picture}
  \end{center}
  \caption{Case 2.1: $|u|>|U|$.}
  \label{del 6}
\end{figure}
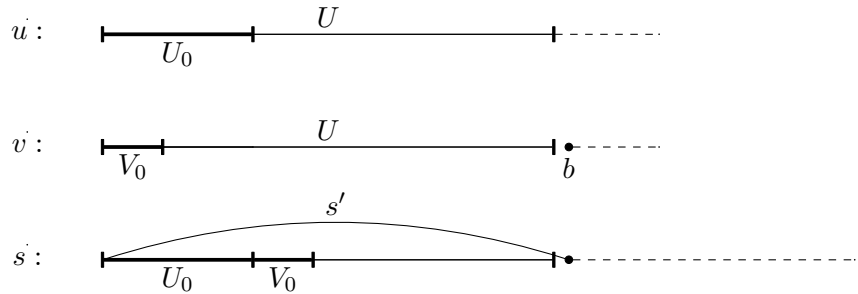
Then $s'$ is the prefix of a proper shuffle $z$ of  $U$  and $Ub$.
Remark that here we mean a prefix of some proper shuffle of $U$ and $Ub$, not
necessarily the same shuffle we have in $s$.
We have $|U|+|Ub|<|u|+|v|$, $U$ is minimal and $Ub\in\mathcal{C}(x)$.
Hence by induction hypothesis $z\sllex Ub$, and hence $s'\sllex Ub$ as $s'$ is the prefix of length $|Ub|$ of $z$.
Thus $s\sllex v$.\\

\noindent{\bf Case 2.2.}
Suppose $|u|\leq |U|$, i.e, $u$ is a prefix of $U$.
Let $r$ denote the longest unbordered prefix of $u$.
Suppose first that  $|r|>|U_0|$, then we can write $r=U_0r'$.
Let $s'$ be such that $|s'|=|r'|$ and
$U_0s'$ is a prefix of $s$ (see Figure~\ref{del 4}).
  \begin{figure}[htbp]
  \begin{center}
\begin{picture}(90,45)(0,5)
\gasset{Nw=.3,Nh=2,Nmr=0,fillgray=0} 
\gasset{ExtNL=y,NLdist=1,NLangle=-90} 

 \node(Y0)(0,40){$$}
  \node(Y1)(20,40){$$}
 \node(Y2)(38,40){$$}
   \node(W0)(0,10){$$}
   \node(W1)(20,10){$$}
   \node(W2)(62,10){$$}
\node(Z0)(0,  25){$$}
\node(Z1)(20,25){$$}
\node(Z2)(60,25){$$}
 \node(Y12)(30,  40){$$}
    \node(Z12)(30,  25){$$}
    \node(W12)(30,  10){$$}

\gasset{Nw=1,Nh=1,Nmr=1,fillgray=0} 

\node(Z3)(62,25){$b$}

  \gasset{Nw=0,Nh=0,Nmr=0,fillgray=0}

  \node(w)(-10,12){$s:$}
  \node(y)(-10,42){$u:$}
  \node(z)(-10,27){$v:$}
\node(W3)(100,10){$$}

  \drawedge[AHnb=0,ELside=r, linewidth=0.5](W0,W1){$U_0$}
   \drawedge[AHnb=0,ELside=l](W1,W2){}
   \drawedge[AHnb=0,ELside=l,dash={1.}](W2,W3){}
  \drawedge[AHnb=0,ELside=r](Z0,Z1){$U_0$}
  \drawedge[AHnb=0,ELside=r, linewidth=0.5](Y0,Y1){$U_0$}
  \drawedge[AHnb=0,ELside=l](Y1,Y2){}
  \drawedge[AHnb=0,ELside=r](Y1,Y12){$r'$}

 \drawedge[AHnb=0,ELside=l](Z0,Z2){}
 \drawedge[AHnb=0,ELside=r](Z1,Z12){$r'$}

\drawedge[AHnb=0,ELside=r](W1,W12){$s'$}

          \drawedge[AHnb=0,ELside=l](W1,W2){}

          \node(Z4)(74,25){$$}
    \drawedge[AHnb=0,ELside=l,dash={1.}](Z3,Z4){}

\gasset{curvedepth=5}

          \drawedge[AHnb=0,ELside=l](Y0,Y12){$r$}
           \drawedge[AHnb=0,ELside=l](Z0,Z12){$r$}
\end{picture}
\end{center}
  \caption{Case 2.2: $u$ is a prefix of $U$ and $|r|>|U_0|$.}
  \label{del 4}
\end{figure}
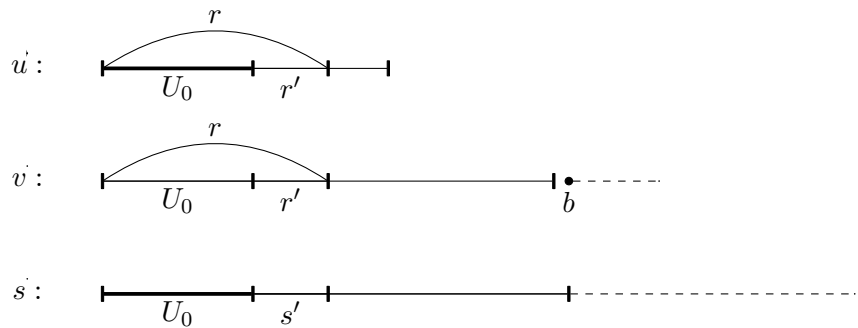
Then $s'$ is 
a prefix of a proper shuffle $z$ of $r$ (which is minimal)
and $r'$ (which is in ${\mathcal{C}}(x)$ since $r$ is unbordered).
By induction hypothesis $z\sllex r'$. This implies $s'\sllex r'$ as $s'$ is the prefix of length $|r'|$ of $z$.
Thus $s\sllex v$.

Thus we can assume that $|r|\leq |U_0|$.
In this case let $v'$ be such that $v=rv'$ and $u'$  such that $u=ru'$.
By Lemma~\ref{lup},  $r$ is a period of $u$ and hence $u'$ is also a prefix of $u$ and hence minimal.
Let $s'$ be such that $rs'$ is a prefix of $s$ of length $|Ub|$ (see Figure~\ref{del 5}).
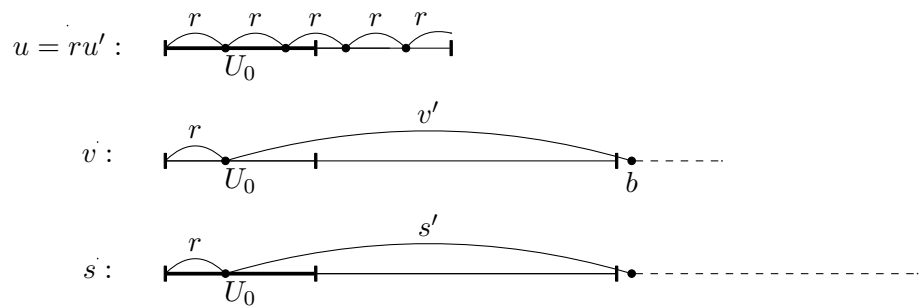
\begin{figure}[htbp]
  \begin{center}
\begin{picture}(102,45)(-15,5)
\gasset{Nw=.3,Nh=2,Nmr=0,fillgray=0} 
\gasset{ExtNL=y,NLdist=1,NLangle=-90} 

 \node(Y0)(0,40){$$}
  \node(Y1)(20,40){$$}
 \node(Y2)(38,40){$$}
 \node(Z0)(0,  25){$$}
\node(Z1)(20,25){$$}
\node(Z2)(60,25){$$}
   \node(W0)(0,10){$$}
   \node(W1)(20,10){$$}
   \node(W2)(60,10){$$}

\gasset{Nw=1,Nh=1,Nmr=1,fillgray=0} 

\node(V1)(8,40){$$}
\node(V2)(16,40){$$}
\node(V3)(24,40){$$}
\node(V4)(32,40){$$}
\node(Z3)(62,25){$b$}
\node(WV1)(8,10){$$}
\node(ZV1)(8,25){$$}
\node(W3)(62,10){$$}

  \gasset{Nw=0,Nh=0,Nmr=0,fillgray=0}

  \node(w)(-9,12){$s:$}
  \node(y)(-13,43){$u=ru':$}
  \node(z)(-9,27.5){$v:$}
 \node(V5)(38,42){$$}
  \node(W4)(100,10){$$}

  \drawedge[AHnb=0,ELside=r, linewidth=0.5](W0,W1){$U_0$}
   \drawedge[AHnb=0,ELside=l](W1,W2){}
   \drawedge[AHnb=0,ELside=l,dash={1.}](W3,W4){}
  \drawedge[AHnb=0,ELside=r](Z0,Z1){$U_0$}
  \drawedge[AHnb=0,ELside=r, linewidth=0.5](Y0,Y1){$U_0$}
  \drawedge[AHnb=0,ELside=l](Y1,Y2){}
  \drawedge[AHnb=0,ELside=r](Y1,Y12){}
\drawedge[AHnb=0,ELside=l](Z0,Z2){}
       \drawedge[AHnb=0,ELside=l](W1,W2){}

       \node(Z4)(74,25){$$}
    \drawedge[AHnb=0,ELside=l,dash={1.}](Z3,Z4){}

    \gasset{curvedepth=1}
           \drawedge[AHnb=0,ELside=l](V4,V5){$r$}

\gasset{curvedepth=2}

          \drawedge[AHnb=0,ELside=l](Y0,V1){$r$}
           \drawedge[AHnb=0,ELside=l](V1,V2){$r$}
           \drawedge[AHnb=0,ELside=l](V2,V3){$r$}
            \drawedge[AHnb=0,ELside=l](V3,V4){$r$}
             \drawedge[AHnb=0,ELside=l](W0,WV1){$r$}
             \drawedge[AHnb=0,ELside=l](Z0,ZV1){$r$}
      \gasset{curvedepth=4}
             \drawedge[AHnb=0,ELside=l](WV1,W3){$s'$}
             \drawedge[AHnb=0,ELside=l](ZV1,Z3){$v'$}
 \end{picture}
  \end{center}
  \caption{Case 2.2: $u$ is a prefix of $U$ and $|r|\leq|U_0|$.}
  \label{del 5}
\end{figure}
Then $s'$ is a prefix of a proper shuffle $z$ of $u'$ and $Ub$ and
hence by induction hypothesis $z\sllex Ub$. As $Ub=rv'\sllex v'$ by Lemma~\ref{period} we obtain $z\sllex v'$.
This implies $s'\sllex v'$ as $s'$ is the prefix of length $|v'|$ of $z$.
So, we get $s\sllex v$ as required. This concludes the proof of Lemma~\ref{propmin}.
\end{proof}

\begin{proof}[Proof of Theorem~\ref{extremal}]
Suppose $y\in \A^\nats$ is Lyndon relative to some order $\leq$ on $\A$.
We will apply the previous lemmas to $x=y$.
Recall that every prefix of $y$ is minimal relative to $\llex$.
We first consider the case where $z=y, $ i.e, when $w\in \sh(y,y)$.
Set
\[
w=\prod _{i=0}^\infty (U_iV_i)
\]
where
\[
y=\prod _{i=0}^\infty  U_i=\prod _{i=0}^\infty  V_i
\]
with each $U_i$ and $V_i$ non-empty.
Since $y$ is Lyndon, and hence in particular not purely periodic, it follows from Lemma~\ref{lup}
that $y$ contains arbitrarily long unbordered prefixes.
Let $v$ be an unbordered prefix of $y$ with $|v|>|U_0|$.
Writing $v=U_0v'$, since $v$ is unbordered and $y$ is Lyndon,
$v'\in {\mathcal{C}}(y)$. Let $s$ be such that $U_0s$ is a prefix of $w$ and $|s|=|v'|$.
Then 
$s$ is a prefix of a proper shuffle $x$ of $v'$ and a prefix of $y$.
By Lemma~\ref{propmin} we deduce that $x\sllex v'$. Then $s\sllex v'$ as $s$ is the prefix of length $|v'|$ of $x$.
Whence $U_0s\sllex U_0v'= v$ and hence $w \sllex y$.

Next suppose $z\neq y$.  Let
\[
w=\prod _{i=0}^\infty (U_iV_i)
\]
where
\[
y=\prod _{i=0}^\infty  U_i\quad\mbox{ and } \quad z=\prod _{i=0}^\infty  V_i
\]
with each $U_i$ and $V_i$ non-empty except for possibly $U_0$.
Suppose first  that $U_0$ is non-empty, that is to say  $y$ dishes out the initial segment of $w$.
Let $r$ be the shortest non-minimal prefix of $z$.
Then by Lemma~\ref{prodmin}, $r\in {\mathcal{C}}(y)$.
Let $t$ denote the prefix of $w$ of length $|r|$.
If $|r|\le|U_0|$ then $t\sllex r$. Suppose that $|r|>|U_0|$. Let $u$ be the prefix of length $|r|$ of $y$.
Then $t$ is a prefix of a proper shuffle $x$ of $u$ and $r$.
By Lemma~\ref{propmin},  $x\sllex r$, whence $t\sllex r$.
So in both cases $w \sllex z$.
Finally suppose that $U_0$ is empty, so that $z$ dishes out the  initial segment of $w$.
Let $z'$ be a tail of $z$ so that $z =V_0z'$.
Writing $w=V_0w'$, it follows that $w'$ is a shuffle of $z'$ and $y$
in which $y$ dishes out the initial segment of $z '$.
If $z' \neq y$ then we are done by the preceding case,
i.e, $w '\sllex z'$ whence $w \sllex z$.
If $z'=y$, then as we saw in the beginning of the proof, we again have $w' \sllex y=z'$ whence $w \sllex z$.
\end{proof}

\begin{remark}
Let $x=11010011001011010010110\cdots$ denote the first shift of
the Thue-Morse infinite word. Let us first show that $x$ is
Lyndon. In \cite{berstel} it was proved that the Thue-Morse word
$\mathbf{T}$ is the lexicographically maximal overlap-free word
beginning with $0$. A word is overlap-free, if it does not contain
factors of the form $avava$, $a\in \A$, $v\in \A^*$. Now, assume
that $x$ is not Lyndon, i.e., there is a tail $y$ of $x$ such that
$x \sllex y$. Then considering the left extension $01^k y$ of $y$
till the closest $0$ in the Thue-Morse word, we get that
$\mathbf{T}\sllex 01^k y$, a contradiction. So, $x$ is Lyndon and
hence is not self-shuffling; yet it can be verified that $x$
begins in only a finite number of Abelian unbordered words.
\end{remark}

Now we prove the following fact about Lyndon words:

\begin{proposition}\label{prop_lyndon} Let $x\in \A^\nats$ be aperiodic. Then $x$ admits a Lyndon word in its shift orbit closure.
\end{proposition}

\begin{proof} We first note that the condition that $x$ is aperiodic is necessary. For instance, if $x=(10)^\omega$ or
$x=10(1100)^\omega,$ then in each case the shift orbit closure of
$x$ does not contain a Lyndon word. So suppose $x\in \A^\nats$ is
aperiodic and let $\leq$ be any order on $\A.$ First suppose $x$
admits an aperiodic uniformly recurrent word $z$ in its shift
orbit closure. For each $n\geq 1$ let $z_n^<$ denote the
lexicographically smallest factor of $z$ of length $n.$ Then for
each $n$ we have that $z_n^<$ is a prefix of $z_{n+1}^<$. Let
$y=\lim_{n\rightarrow \infty}z_n^<.$ Then $y$ is in the shift
orbit closure of $z$ and hence in the shift orbit closure of $x.$
Moreover since $z$ is both aperiodic and uniformly recurrent it
follows that $y$ is Lyndon.

So we may suppose that every uniformly recurrent word $z$ in the
shift orbit closure of $x$ is ultimately periodic, and hence
purely periodic. For each $n\geq 1$ let $x_n^<$ and $x_n^>,$
respectively, denote the lexicographically smallest and largest,
respectively, factor of $x$ of length $n.$ We claim that for some
positive integer $n,$ either $x_n^<$ or $x_n^>$ is not uniformly
recurrent in $x$ (i.e., does not occur with bounded gaps). In
fact, suppose to the contrary that both $x_n^<$ and $x_n^>$ are
uniformly recurrent in $x$ for each $n\geq 1.$ Let
$y^<=\lim_{n\rightarrow \infty}x_n^<$ and $y^>=\lim_{n\rightarrow
\infty}x_n^>.$ Since each $x_n^<$ is a prefix of $y^<,$ and since
each $x_n^<$ is uniformly recurrent in $y^<,$ it follows that
$y^<$ is uniformly recurrent and hence periodic. Similarly, we
deduce $y^>$ is also periodic.  Again, since each $x_n^<$ and
$x_n^>$ is uniformly recurrent in $x,$ it follows that $y^<$ and
$y^>$ have the same set of factors. Hence, each is a suffix of the
other. Fix a prefix $v$ of $y^<$ so that $y^<=vy^>.$  Now since
$x$ is aperiodic, there exist a prefix $u$ of $y^<$ and distinct
letters $a,b \in \A$ with $|u|>|v|,$ $ua$ a prefix of $y^<$ and
$ub$ a factor of $x.$ Since $ua$ is a prefix of $y^<,$ we deduce
that $ua<ub$ and hence that $a<b.$ On the other hand, since
$v^{-1}ua$ is a prefix of $y^>,$ we deduce that
$v^{-1}ua>v^{-1}ub$ and hence that $a>b.$ This contradiction
establishes the desired claim.

Fix a positive integer $n$ such that either $x_n^<$ or $x_n^>$ is
not uniformly recurrent in $x.$ Up to replacing $\leq$ by the
opposite order, we can assume that $x_n^<$ is not uniformly
recurrent in $x.$ It follows that the shift orbit closure of $x$
contains an infinite word $y$ such that i) $x_n^<$ is a prefix of
$y$ and ii) $x_n^<$ is uni-occurrent in $y.$ It follows
immediately that $y$ is Lyndon.
\end{proof}

As an immediate consequence of Theorem \ref{Lyn} and Proposition
\ref{prop_lyndon} we have

\begin{corollary}
Every aperiodic word $x$  contains a point $y$ in its shift orbit
closure  which is not self-shuffling.
\end{corollary}

\section{A characterization of self-shuffling Sturmian words}\label{Sturmy}

Sturmian words admit various types of
characterizations of geometric and combinatorial nature, e.g., they can be defined via balance,
complexity, morphisms, etc. (see Chapter 2 in \cite{Lo2}).
In \cite {MoHe2}, Morse and Hedlund showed that each Sturmian word
may be realized geometrically by an irrational rotation on the
circle. More precisely, every Sturmian word $x$ is obtained by
coding the symbolic orbit of a point $\rho(x)$ on the circle (of
circumference one) under a rotation by an irrational angle $\alpha$
where the circle is partitioned into two complementary
intervals, one of length $\alpha $ (labeled $1)$ and the other of
length $1-\alpha $ (labeled $0)$ (see Fig.~\ref{geosturm}).
And conversely each such coding gives rise to a Sturmian word.
The irrational $\alpha $ is called the {\it slope} and the point
$\rho(x)$ is called the {\it intercept} of the Sturmian word $x$.
A Sturmian word $x$ of slope $\alpha$ with $\rho(x)=\alpha$ is called a {\it characteristic Sturmian word}.
It is well known that every prefix $u$ of a characteristic Sturmian word is {\it left
special,} i.e., both $0u$ and $1u$ are factors of $x$ \cite{Lo2}.
Thus if $x$ is a characteristic Sturmian word of slope $\alpha$,
then both $0x$ and $1x$ are Sturmian words of slope $\alpha$ and $\rho(0x)=\rho(1x)=0$.
The fact that $\rho$ is not one-to-one stems from the ambiguity of the coding of the boundary points $0$ and $1-\alpha$.

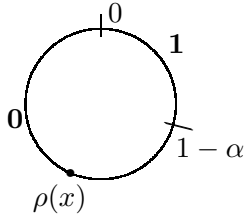
\begin{figure}[htbp]
\unitlength=.5mm
\begin{picture}(20,60)
\put(140,0){
  \qbezier(25.000,10.000)(33.284,10.000)(39.142,15.858)
  \qbezier(39.142,15.858)(45.000,21.716)(45.000,30.000)
  \qbezier(45.000,30.000)(45.000,38.284)(39.142,44.142)
  \qbezier(39.142,44.142)(33.284,50.000)(25.000,50.000)
  \qbezier(25.000,50.000)(16.716,50.000)(10.858,44.142)
  \qbezier(10.858,44.142)( 5.000,38.284)(5.000,30.000)
  \qbezier( 5.000,30.000)( 5.000,21.716)(10.858,15.858)
  \qbezier(10.858,15.858)(16.716,10.000)(25.000,10.000)

  \put(0,23.75){$\mathbf 0$}
  \put(43,43.75){$\mathbf 1$}
  \put(17,11.67){\circle*{2}}  \put(7,3){$\rho(x)$}

  \put(25,48){\line(0,1){6}}  \put(27,52){$0$}
  \put(42,25){\line(4,-1){7.5}} \put(45,16){$1-\alpha$}}
\end{picture}
\caption{Geometric picture of a Sturmian word of slope $\alpha$.}
\label{geosturm}
\end{figure}
\smallskip

\begin{theorem}\label{the:rotation}
Let $S$, $M$ and $L$ be Sturmian words of the same slope $\alpha$, $0<\alpha<1$, satisfying $S\le_{\lex}M\le_{\lex}L$.
Then $M\in\sh (S,L)$ if and only if the following conditions hold:
If  $\rho(M)=\rho(S)$ (respectively, $\rho(M)=\rho(L))$,
then $\rho(L)\neq 0$ (respectively $\rho(S)\neq 0$).
\end{theorem}

Before proving the theorem, we mention two corollaries. In
particular (taking $S=M=L)$, we obtain the following
characterization of self-shuffling Sturmian words:

\begin{corollary}\label{ssSturm}
A Sturmian word $x\in \{0,1\}^\nats$ is self-shuffling if and only if $\rho(x)\neq 0$,
or equivalently, $x$ is not of the form $aC$ where $a\in \{0,1\}$ and $C$ is a characteristic Sturmian word.
\end{corollary}

As another immediate consequence we have:

\begin{corollary}
Let $C\in \{0,1\}^\nats$ be a characteristic Sturmian word beginning in $a^nb$ with  $\{a,b\}=\{0,1\}$.
Then every point $x$ in the shift orbit closure of $C$ beginning in $a^n$ belongs to $\sh (C,C)$.
\end{corollary}

\begin{proof} Let $X$ denote the shift orbit closure of $C$. In case $x=C$ the result follows from Corollary~\ref{ssSturm}.
Next suppose $x\neq C$. Suppose $bx\in X$.
Then  by Theorem~\ref{the:rotation} it follows that $bx\in \sh(bC,C)$ whence $x\in \sh(C,C)$.
Next suppose $ax\in X$.
Then again by Theorem~\ref{the:rotation} we have $ax\in \sh (aC,C)$.
Since $ax$ begins in $a^{n+1}$, then for any shuffle of $aC$ and $C$
which produces $ax$, one of the initial $n+1$ leading $a'$s in $ax$ must come from the leading $a$ in $aC$.
Thus it can always be arranged that the first $a$ in $ax$ comes from the leading $a$ in $aC$.
Thus again we have that $x\in \sh (C,C)$.
\end{proof}

\begin{proof}[Proof of Theorem~\ref{the:rotation}]
We begin by showing that the conditions stated in Theorem~\ref{the:rotation} are in fact necessary for $M\in\sh(S,L)$.
To see this, suppose $\rho(M)=\rho(S)$ and $\rho(L)=0$ (the other symmetric condition works analogously).
This implies that $L \in \{0x,1x\}$ where $x$ is the characteristic Sturmian word of slope $\alpha$.
If $L=0x$, then as $0x$ is minimal in the Sturmian subshift of slope $\alpha$, it follows that $S=M=L$.
Whence by Proposition~\ref{abborders},  $M\notin \sh(M,M)=\sh(S,L)$. If $L=1x$,
we consider the lexicographic order induced by $0>1$.
Then $L\le_{\lex}M\le_{\lex}S$ and moreover $L$ is minimal.
Since $\rho(M)=\rho(S)$ we have that either case i) $M=S$ or case ii) $S=0x$ and $M=1x$
or case iii) there exists $u\in \{0,1\}^*$ such that $S=u01x$ and $M=u10x$
where in each case $x$ denotes the characteristic Sturmian word of slope $\alpha$.
In case i), using Theorem~\ref{extremal} we deduce that each element of $\sh (S,L)$ is lexicographically smaller
than $S$ and hence since $M=S$ we have $M\notin \sh (S,L)$.
In case ii), if $M\in \sh(S,L)$, then $x\in \sh (x,0x)$ which contradicts Theorem~\ref{extremal}.
Finally, in case iii), suppose to the contrary that $M\in \sh(S,L)$.
Then since $u0$ is not a prefix of $M$, it follows that there exists a non-empty prefix $v$ of $L$
and a prefix $w$ of  $M$ such that $|w|=|u0|+|v|$ and $M'\in \sh (L', 1x)$
where $M'$ and $L'$ are defined by $M=wM'$ and $L=vL'$.
But this implies $M'=L'$, whence $L'\in \sh(L',1x)$ which again contradicts Theorem~\ref{extremal}.\\

We next show that the conditions stated in Theorem~\ref{the:rotation} are sufficient.
Without loss of generality we can assume $0<\alpha<1/2$.

Our proof explicitly describes an algorithm for shuffling $S$ and $L$ so as to produce $M$.
It is formulated in terms of the circle rotation description of Sturmian words.
Geometrically speaking, points $\rho(S)$ and $\rho(L)$ will take turns following the trajectory of $\rho(M)$
so that the respective codings agree; as one follows the other waits its turn (remains neutral).
The algorithm specifies this following rule depending on the relative positions
of the trajectories of all three points and is broken down into several cases.
The proof can be summarized by the directed graph  in Fig.~\ref{pc-fig} in which each state $n$
corresponds to ``case $n"$ in the proof.\\
\begin{figure}[htbp]
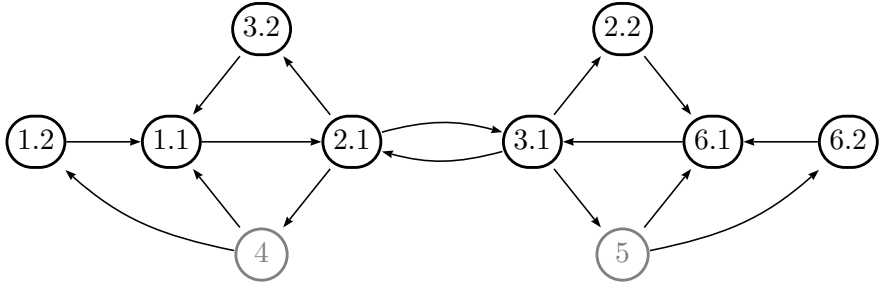

\centering
\VCDraw{%
\begin{VCPicture}{(0,-3)(20,3)}
 \LargeState
 \StateVar[1.2]{(1,0)}{1.2}
 \StateVar[1.1]{(4,0)}{1.1}
 \StateVar[2.1]{(8,0)}{2.1}
 \StateVar[3.2]{(6,2.5)}{3.2}
 \StateVar[3.1]{(12,0)}{3.1}
 \StateVar[6.1]{(16,0)}{6.1}
  \StateVar[6.2]{(19,0)}{6.2}
\StateVar[2.2]{(14,2.5)}{2.2}
 \DimState \StateVar[4]{(6,-2.5)}{4}
 \StateVar[5]{(14,-2.5)}{5}
\EdgeR{1.2}{1.1}{}
\EdgeR{1.1}{2.1}{}
\EdgeR{2.1}{3.2}{}
\EdgeR{2.1}{4}{}
\EdgeR{3.2}{1.1}{}
\EdgeR{4}{1.1}{}
\ArcL{2.1}{3.1}{}
\ArcL{3.1}{2.1}{}
\EdgeR{6.2}{6.1}{}
\EdgeR{6.1}{3.1}{}
\EdgeR{3.1}{2.2}{}
\EdgeR{3.1}{5}{}
\EdgeR{2.2}{6.1}{}
\EdgeR{5}{6.1}{}
\ArcL{4}{1.2}{}
\ArcR{5}{6.2}{}
\end{VCPicture}}
\caption{Graphical depiction of the proof of Theorem~\ref{the:rotation}.}
\label{pc-fig}
\end{figure}

\noindent{\bf The states:}
We denote by $s$, $m$, and $\ell$ the current tail of the words $S$, $M$, and $L$.
They are initialized as
\[
s:=S, \ \ell:=L, \text{ and } m:=M.
\]
While $m$ is always a tail of $M$, the letters $s$ and $\ell$ may be tails of $S$ or $L$,
depending on which is the current lexicographically largest\footnote{The choice of the letter $s,m$,
and $\ell$ is intended to refer to {\em small, medium}, and {\em large} respectively.}.
Each state, or case in the proof, is described by a figure depicting
the relative positions of $\rho(s)$, $\rho(m)$ and $\rho(\ell)$,
which for the sake of simplicity, are actually labeled $s,m$ and $\ell$ respectively.
If $x\in\{s,m,\ell\}$  is depicted inside the interval $(0,1-\alpha)$ (resp. $(1-\alpha,1)$),
then this implies that the first letter of the coding of $x$ is $0$ (respectively $1$).
Moreover the endpoints of the partition interval are regarded as both open and closed.
For example, even if $s$ and $m$ are both depicted in the interval $(0,1-\alpha)$,
it could be that  $\rho(s)=0$ and $\rho(m)=1-\alpha$. In the same way, even if $s$ and $m$
are depicted in distinct intervals of the circle partition, it could be that $\rho(s)=\rho(m)$.
In addition to their relative positions on the circle, each state lists a set of relations
which the variables $s$, $m$ and $\ell$ must satisfy.
These conditions are written to the right of the circle picture and are described in terms of the following predicates:
\begin{align*}
C(s,m,\ell) & \equiv \  [\rho(m)=\rho(s) \text{ and } \rho(\ell)=0] \text{ or }  [\rho(m)=\rho(\ell) \text{ and } \rho(s)=0];\\
P_1(s,m,\ell) & \equiv \ \rho(s)=\alpha \text{ and } \rho(m)=0 \text{ and } \rho(\ell)=1- \alpha;\\
P_2(s,m,\ell) & \equiv \  [(\rho(\ell)-\rho(m)) \bmod{1}= \alpha  \text{ and } \rho(s)=1-\alpha]  \text{ or }  \rho(m)=0.
\end{align*}
All states, except those labeled $4$ and $5$, can be taken to be initial states. \\

\noindent{\bf The edges:} Each directed edge corresponds to a precise set of instructions
which specify which of $s$ or $\ell$ is neutral,  which of $s$ or $\ell$ follows $m$
and for how long, and in the end a possible relabeling of the variables $s$ and $\ell$.
In each case the outcome leads to a new case in which there is a switch in the follower.
In other words, if there is an edge from case $i$ to case  $j$ in the graph,
then either the instructions for case $i$ and case $j$ specify different followers
(as is the case for cases $1.1$ and $2.1$) in which case the passage from $i$ to $j$
leaves the labeling of $s$ and $\ell$ unchanged, or the instructions for case $i$ and case $j$
specify the same follower (as is the case for cases $1.2$ and $1.1$)
in which case the passage from $i$ to $j$ exchanges the labeling of $s$ and $\ell$.
The proof of Theorem~\ref{the:rotation} amounts to showing that for each state $n$ in the graph,
the specified instructions will take $n$ to an adjacent state in the graph.\\

What follows is a complete listing of all ten cases with their respective set of instructions.\\

\noindent
\begin{minipage}{\textwidth}
{\bf Case 1.1:}

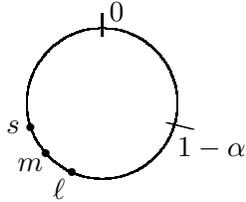
\unitlength=.5mm
\begin{picture}(50,60)
\put(50,0){
  \qbezier(25.000,10.000)(33.284,10.000)(39.142,15.858)
  \qbezier(39.142,15.858)(45.000,21.716)(45.000,30.000)
  \qbezier(45.000,30.000)(45.000,38.284)(39.142,44.142)
  \qbezier(39.142,44.142)(33.284,50.000)(25.000,50.000)
  \qbezier(25.000,50.000)(16.716,50.000)(10.858,44.142)
  \qbezier(10.858,44.142)(5.000,38.284)(5.000,30.000)
  \qbezier( 5.000,30.000)( 5.000,21.716)(10.858,15.858)
  \qbezier(10.858,15.858)(16.716,10.000)(25.000,10.000)

  \put(6,23.75){\circle*{2}}   \put(-0.5,22){$s$}
  \put(10,16.77){\circle*{2}}  \put(2.5,12){$m$}
  \put(17,11.67){\circle*{2}}  \put(12,4.5){$\ell$}

  \put(25,48){\line(0,1){6}}  \put(27,52){$0$}
  \put(42,25){\line(4,-1){7.5}} \put(45,16){$1-\alpha$}}

  \put(150,27.5){$s\le_{\lex} m <_{\lex}\ell$; $\neg C(s,m,\ell)$}
\end{picture}
\end{minipage}
  \smallskip

\noindent {\it Instruction:} $s$ is neutral and $\ell$ follows $m$ until they lie in different elements of the circle partition.
No relabeling of $s$ and $\ell$.\\

\noindent
\begin{minipage}{\textwidth}
{\bf Case 1.2:}

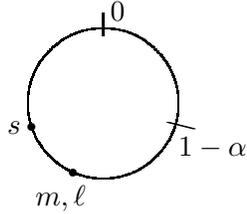
\unitlength=.5mm
\begin{picture}(50,60)
\put(50,0){
  \qbezier(25.000,10.000)(33.284,10.000)(39.142,15.858)
  \qbezier(39.142,15.858)(45.000,21.716)(45.000,30.000)
  \qbezier(45.000,30.000)(45.000,38.284)(39.142,44.142)
  \qbezier(39.142,44.142)(33.284,50.000)(25.000,50.000)
  \qbezier(25.000,50.000)(16.716,50.000)(10.858,44.142)
  \qbezier(10.858,44.142)( 5.000,38.284)(5.000,30.000)
  \qbezier( 5.000,30.000)( 5.000,21.716)(10.858,15.858)
  \qbezier(10.858,15.858)(16.716,10.000)(25.000,10.000)

  \put(6,23.75){\circle*{2}}   \put(-0.5,22){$s$}
  \put(17,11.67){\circle*{2}}  \put(7,3){$m,\ell$}

  \put(25,48){\line(0,1){6}}  \put(27,52){$0$}
  \put(42,25){\line(4,-1){7.5}} \put(45,16){$1-\alpha$}}

  \put(150,27.5){$s\le_{\lex} m =\ell$; $\neg C(s,m,\ell)$}
\end{picture}
\end{minipage}
\smallskip

\noindent {\it Instruction:}
$s$ is neutral and $\ell$ follows $m$ until $0<\rho(m)=\rho(\ell)<\rho(s)$.
We note that this is always possible since $\rho(s)\not =0$
and the set $\left\{(\rho(m)+n\alpha)\bmod{1}\colon n\in\mathbb{N}\right\}$ is dense in the unit circle.
Exchange the labels $s \leftrightarrow \ell$.\\

\noindent
\begin{minipage}{\textwidth}
{\bf Case 2.1:}

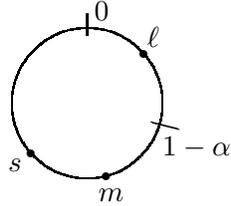
\unitlength=.5mm
\begin{picture}(50,60)
\put(50,0){
  \qbezier(25.000,10.000)(33.284,10.000)(39.142,15.858)
  \qbezier(39.142,15.858)(45.000,21.716)(45.000,30.000)
  \qbezier(45.000,30.000)(45.000,38.284)(39.142,44.142)
  \qbezier(39.142,44.142)(33.284,50.000)(25.000,50.000)
  \qbezier(25.000,50.000)(16.716,50.000)(10.858,44.142)
  \qbezier(10.858,44.142)( 5.000,38.284)(5.000,30.000)
  \qbezier( 5.000,30.000)( 5.000,21.716)(10.858,15.858)
  \qbezier(10.858,15.858)(16.716,10.000)(25.000,10.000)

  \put(10,16.77){\circle*{2}}   \put(4,12){$s$}
  \put(30,10.63){\circle*{2}}   \put(28,4){$m$}
  \put(40,43.23){\circle*{2}}   \put(41,45){$\ell$}

\put(25,48){\line(0,1){6}}  \put(27,52){$0$}
  \put(42,25){\line(4,-1){7.5}} \put(45,16){$1-\alpha$}}

  \put(150,27.5){$s<_{\lex} m$; $\neg C(s,m,\ell)$}
\end{picture}
\end{minipage}
\smallskip

\noindent {\it Instruction:}
$\ell $ is neutral and $s$ follows $m$ until  they lie in different elements of the circle partition.
No relabeling of $s$ and $\ell$.
Three cases are possible according to the relative position of $m$ and $\ell$ in the partition $(1-\alpha, 1)$.\\

\noindent
\begin{minipage}{\textwidth}
{\bf Case 2.2:}

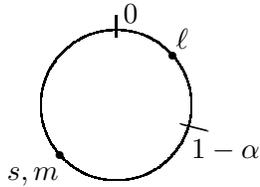
\unitlength=.5mm
\begin{picture}(50,60)
\put(50,0){
  \qbezier(25.000,10.000)(33.284,10.000)(39.142,15.858)
  \qbezier(39.142,15.858)(45.000,21.716)(45.000,30.000)
  \qbezier(45.000,30.000)(45.000,38.284)(39.142,44.142)
  \qbezier(39.142,44.142)(33.284,50.000)(25.000,50.000)
  \qbezier(25.000,50.000)(16.716,50.000)(10.858,44.142)
  \qbezier(10.858,44.142)( 5.000,38.284)(5.000,30.000)
  \qbezier( 5.000,30.000)( 5.000,21.716)(10.858,15.858)
  \qbezier(10.858,15.858)(16.716,10.000)(25.000,10.000)

  \put(10,16.77){\circle*{2}}   \put(-4,10){$s,m$}
  \put(40,43.23){\circle*{2}}   \put(41,45){$\ell$}

\put(25,48){\line(0,1){6}}  \put(27,52){$0$}
  \put(42,25){\line(4,-1){7.5}} \put(45,16){$1-\alpha$}}

  \put(150,27.5){$s= m$; $\neg C(s,m,\ell)$}
\end{picture}
\end{minipage}
\smallskip

\noindent {\it Instruction:}
$\ell$ is neutral and $s$ follows $m$ until $\rho(m)=\rho(s)>\rho(\ell)$.
This is possible because $\rho(\ell)\neq 0$.
Exchange the labels $s \leftrightarrow \ell$.\\

\noindent
\begin{minipage}{\textwidth}
{\bf Case 3.1:}

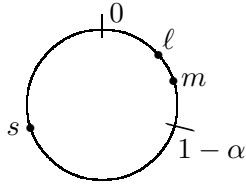
\unitlength=.5mm
\begin{picture}(50,60) 
\put(50,0){
  \qbezier(25.000,10.000)(33.284,10.000)(39.142,15.858)
  \qbezier(39.142,15.858)(45.000,21.716)(45.000,30.000)
  \qbezier(45.000,30.000)(45.000,38.284)(39.142,44.142)
  \qbezier(39.142,44.142)(33.284,50.000)(25.000,50.000)
  \qbezier(25.000,50.000)(16.716,50.000)(10.858,44.142)
  \qbezier(10.858,44.142)( 5.000,38.284)(5.000,30.000)
  \qbezier( 5.000,30.000)( 5.000,21.716)(10.858,15.858)
  \qbezier(10.858,15.858)(16.716,10.000)(25.000,10.000)

  \put(6,23.75){\circle*{2}}   \put(-0.5,22){$s$}
  \put(44,36.245){\circle*{2}}   \put(46,35){$m$}
  \put(40,43.23){\circle*{2}}   \put(41,45){$\ell$}

  \put(25,48){\line(0,1){6}}  \put(27,52){$0$}
  \put(42,25){\line(4,-1){7.5}} \put(45,16){$1-\alpha$}}

  \put(150,27.5){$m<_{\lex}\ell$; $\neg C(s,m,\ell)$}
  \end{picture}
\end{minipage}
  \smallskip

\noindent {\it Instruction:}
$s$ is neutral and $\ell$ follows $m$ until they lie in different elements of the circle partition.
No relabeling of $s$ and $\ell$.
Three cases are possible according to the relative position of $m$ and $s$ in the partition $(0, 1-\alpha)$.\\

\noindent
\begin{minipage}{\textwidth}
{\bf Case 3.2:}

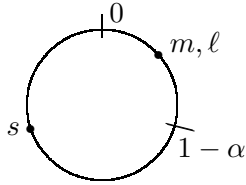
\unitlength=.5mm
  \begin{picture}(50,60) 
\put(50,0){
  \qbezier(25.000,10.000)(33.284,10.000)(39.142,15.858)
  \qbezier(39.142,15.858)(45.000,21.716)(45.000,30.000)
  \qbezier(45.000,30.000)(45.000,38.284)(39.142,44.142)
  \qbezier(39.142,44.142)(33.284,50.000)(25.000,50.000)
  \qbezier(25.000,50.000)(16.716,50.000)(10.858,44.142)
  \qbezier(10.858,44.142)( 5.000,38.284)(5.000,30.000)
  \qbezier( 5.000,30.000)( 5.000,21.716)(10.858,15.858)
  \qbezier(10.858,15.858)(16.716,10.000)(25.000,10.000)

  \put(6,23.75){\circle*{2}}   \put(-0.5,22){$s$}
  \put(40,43.23){\circle*{2}}   \put(43,44){$m,\ell$}

  \put(25,48){\line(0,1){6}}  \put(27,52){$0$}
  \put(42,25){\line(4,-1){7.5}} \put(45,16){$1-\alpha$}}

  \put(150,27.5){$m=\ell$; $\neg C(s,m,\ell)$}
  \end{picture}
\end{minipage}
  \smallskip

\noindent {\it Instruction:}
$s$ is neutral and $\ell$ follows $m$ until $0<\rho(m)=\rho(\ell)<\rho(s)$.
 This is possible because $\rho(s)\neq 0$.
Exchange the labels $s \leftrightarrow \ell$.\\

\noindent
\begin{minipage}{\textwidth}
{\bf Case 4:}

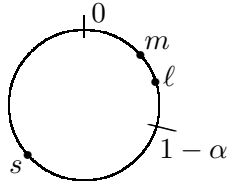
\unitlength=.5mm
\begin{picture}(50,60) 
\put(50,0){
  \qbezier(25.000,10.000)(33.284,10.000)(39.142,15.858)
  \qbezier(39.142,15.858)(45.000,21.716)(45.000,30.000)
  \qbezier(45.000,30.000)(45.000,38.284)(39.142,44.142)
  \qbezier(39.142,44.142)(33.284,50.000)(25.000,50.000)
  \qbezier(25.000,50.000)(16.716,50.000)(10.858,44.142)
  \qbezier(10.858,44.142)( 5.000,38.284)(5.000,30.000)
  \qbezier( 5.000,30.000)( 5.000,21.716)(10.858,15.858)
  \qbezier(10.858,15.858)(16.716,10.000)(25.000,10.000)

  \put(10,16.77){\circle*{2}}   \put(5,12){$s$}
  \put(44,36.245){\circle*{2}}   \put(46,35){$\ell$}
  \put(40,43.23){\circle*{2}}   \put(41,45){$m$}

   \put(25,48){\line(0,1){6}}  \put(27,52){$0$}
  \put(42,25){\line(4,-1){7.5}} \put(45,16){$1-\alpha$}}

  \put(150,27.5){$m>_{\lex}\ell$; $\rho(s)\ge\alpha$; $\neg P_1(s,m,\ell)$}
  \end{picture}
\end{minipage}
  \smallskip

\noindent {\it Instruction:}
$s$ is neutral and $\ell$ follows $m$ for just one rotation by $\alpha$.
Exchange the labels $s \leftrightarrow \ell$.
Because  $\rho(s)\ge\alpha$, we have either $m<_{\lex}s$ or $m=s$.\\

\noindent
\begin{minipage}{\textwidth}
{\bf Case 5:}

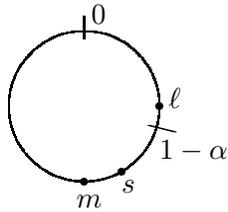
\unitlength=.5mm
\begin{picture}(50,60)
\put(50,0){
\qbezier(25.000,10.000)(33.284,10.000)(39.142,15.858)
  \qbezier(39.142,15.858)(45.000,21.716)(45.000,30.000)
  \qbezier(45.000,30.000)(45.000,38.284)(39.142,44.142)
  \qbezier(39.142,44.142)(33.284,50.000)(25.000,50.000)
  \qbezier(25.000,50.000)(16.716,50.000)(10.858,44.142)
  \qbezier(10.858,44.142)( 5.000,38.284)(5.000,30.000)
  \qbezier( 5.000,30.000)( 5.000,21.716)(10.858,15.858)
  \qbezier(10.858,15.858)(16.716,10.000)(25.000,10.000)

  \put(25,10){\circle*{2}}   \put(23,3){$m$}
  \put(35,12.68){\circle*{2}}   \put(35,7){$s$}
  \put(45,30){\circle*{2}}   \put(47.5,29.5){$\ell$}

\put(25,48){\line(0,1){6}}  \put(27,52){$0$}
  \put(42,25){\line(4,-1){7.5}} \put(45,16){$1-\alpha$}}

  \put(150,27.5){$m<_{\lex}s$; $(\rho(\ell)-\rho(m))\bmod{1} \le\alpha$; $\neg P_2(s,m,\ell)$}
\end{picture}
\end{minipage}
\smallskip

 \noindent {\it Instruction:}
$\ell$ is neutral and $s$ follows $m$ for just one rotation by $\alpha$.
 Exchange the labels $s \leftrightarrow \ell$.
Because $(\rho(\ell)-\rho(m))\bmod{1} \le\alpha$,  we have either $m>_{\lex}\ell$ or $m=\ell$.\\

\noindent
\begin{minipage}{\textwidth}
{\bf Case 6.1:}

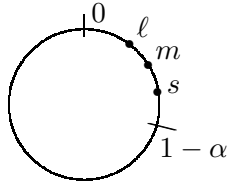
\unitlength=.5mm
\begin{picture}(50,60)  
  \put(50,0){
\qbezier(25.000,10.000)(33.284,10.000)(39.142,15.858)
  \qbezier(39.142,15.858)(45.000,21.716)(45.000,30.000)
  \qbezier(45.000,30.000)(45.000,38.284)(39.142,44.142)
  \qbezier(39.142,44.142)(33.284,50.000)(25.000,50.000)
  \qbezier(25.000,50.000)(16.716,50.000)(10.858,44.142)
  \qbezier(10.858,44.142)( 5.000,38.284)(5.000,30.000)
  \qbezier( 5.000,30.000)( 5.000,21.716)(10.858,15.858)
   \qbezier(10.858,15.858)(16.716,10.000)(25.000,10.000)

\put(44.5,33.44){\circle*{2}}   \put(47,33){$s$}
  \put(42,40.53){\circle*{2}}   \put(44,42){$m$}
  \put(37,46){\circle*{2}}   \put(39,48){$\ell$}

\put(25,48){\line(0,1){6}}  \put(27,52){$0$}
  \put(42,25){\line(4,-1){7.5}} \put(45,16){$1-\alpha$}}

  \put(150,27.5){$s<_{\lex}m\le_{\lex}\ell$; $\neg C(s,m,\ell)$}
\end{picture}
\end{minipage}
\smallskip

\noindent {\it Instruction:}
$\ell$ is neutral and $s$ follows $m$ until they lie in different elements of the circle partition.
No relabeling of $s$ and $\ell$.\\

\noindent
\begin{minipage}{\textwidth}
{\bf Case 6.2:}

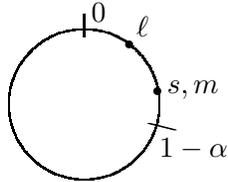
\unitlength=.5mm
\begin{picture}(50,60)  
  \put(50,0){
\qbezier(25.000,10.000)(33.284,10.000)(39.142,15.858)
  \qbezier(39.142,15.858)(45.000,21.716)(45.000,30.000)
  \qbezier(45.000,30.000)(45.000,38.284)(39.142,44.142)
  \qbezier(39.142,44.142)(33.284,50.000)(25.000,50.000)
  \qbezier(25.000,50.000)(16.716,50.000)(10.858,44.142)
  \qbezier(10.858,44.142)( 5.000,38.284)(5.000,30.000)
  \qbezier( 5.000,30.000)( 5.000,21.716)(10.858,15.858)
   \qbezier(10.858,15.858)(16.716,10.000)(25.000,10.000)

\put(44.5,33.44){\circle*{2}}   \put(47,33){$s,m$}
  \put(37,46){\circle*{2}}   \put(39,48){$\ell$}

\put(25,48){\line(0,1){6}}  \put(27,52){$0$}
  \put(42,25){\line(4,-1){7.5}} \put(45,16){$1-\alpha$}}

  \put(150,27.5){$s=m\le_{\lex}\ell$; $\neg C(s,m,\ell)$}
\end{picture}
\end{minipage}
\smallskip

\noindent {\it Instruction:}
$\ell$ is neutral and $s$ follows $m$ until $\rho(m)=\rho(s)>\rho(\ell)$.
 This is possible because $\rho(\ell)\neq 0$.
  Exchange the labels $s \leftrightarrow \ell$.\\

Here we verify four of the ten cases in the proof of Theorem~\ref{the:rotation}.
The verifications of all cases are similar to one another.\\

\noindent {\bf Verification of Case 1.1:}
To see that Case 1.1 leads to Case 2.1, let $m'$ and $\ell '$ denote the positions of
$m$ and $\ell$ respectively, the first time  they lie in different elements of the circle partition.
Then clearly $0\leq\rho(s) < \rho(m') \leq 1-\alpha \leq \rho(\ell')$ as required.
It remains to show that after the rotation $\neg C(s, m',\ell')$ holds.
Suppose to the contrary that $C(s, m',\ell')$ holds.
Because $\rho(m')>\rho(s)$, we must have $\rho(m')=\rho(\ell')$ and $\rho(s)=0$.
But this implies $\rho(m)=\rho(\ell)$ and
$\rho(s)=0$, which is impossible since we had assumed $\neg C(s,m,\ell)$.\\

\noindent {\bf Verification of Case 1.2:}
To see that Case 1.2 leads to Case 1.1, let  $m'$ and $\ell '$
denote the positions of $m$ and $\ell$ respectively, the first time $0<\rho(m')=\rho(\ell')<\rho(s)$.
Then clearly after exchanging the labels $s\leftrightarrow \ell$ the points $s,m$,
and $\ell$ are situated as specified in Case 1.1.
It remains to show  that $\neg C(\ell',m',s)$.
Suppose to the contrary that $C(\ell', m',s)$ holds.
Because $m'=\ell'$, we must have $\rho(m')=\rho(\ell')$ and $\rho(s)=0$.
But this implies $\rho(m)=\rho(\ell)$ and $\rho(s)=0$,
which is impossible since we had assumed $\neg C(s,m,\ell)$.\\

\noindent {\bf Verification of Case 2.1:}
Let $s'$ and $m'$ denote the positions of $s$ and $m$ respectively,
the first time  they lie in different elements of the circle partition.
Note that since we have assumed $\alpha<1/2$, it follows that $\rho(s')\ge\alpha$ for otherwise $m$ and $s$
would have already differed earlier.
Three cases are possible: $m'<_{\lex}\ell$, $m'=\ell$ or $m'>_{\lex}\ell$.
We show that this leads to cases 3.1,  3.2 and 4 respectively.
Assume first that $m'\le_{\lex}\ell$.
To show that this leads to Case 3.1 or Case 3.2, we must verify that $\neg C(s',m',\ell)$.
However we have $\alpha\le \rho(s')\le 1-\alpha$, and hence $\rho(s')\neq 0$.
 If $\rho(m')=\rho(s')$, then $\rho(m)=\rho(s)$, and hence $\rho(\ell)\neq 0$
 since we had assumed $\neg C(s,m,\ell)$.
Next we suppose that $m'>\ell$.
To show this result in Case 4, we must show that
 \[
\neg P_1(s',m',\ell).
\]
Assume to the contrary that $P_1(s',m',\ell)$, that is,
 that $\rho(s')=\alpha$, $\rho(m')=0$ and $\rho(\ell)=1-\alpha$.
 This implies $m=0m'$ and $s=0s'$, and hence $\rho(m)=1-\alpha=\rho(\ell)$ and $\rho(s)=0$,
 which is impossible since we had assumed $\neg C(s,m,\ell)$.\\

\noindent {\bf Verification of Case 4}:
Let $\ell'$ and $m'$ denote the positions of $\ell$ and $m$ after rotation by $\alpha$,
Because  $\rho(s)\ge\alpha$, we have either $m'<_{\lex}s$ or $m'=s$.
We will show that this leads to cases 1.1 and 1.2 respectively.
In view of the label exchange $s\leftrightarrow \ell$,
the relative positions of the three points is as required.
It remains to check in both cases that  $\neg C(\ell',m',s)$ holds.
  Since $\rho(s)\ge\alpha$, it follows that $\rho(s)\neq 0$.
  If $\rho(m')=\rho(s)$, then we actually have $\rho(m')=\rho(s)=\alpha$.
  This implies $\rho(m)=0$. Because we had assumed $\neg P_1(s,m,\ell)$,
  we obtain $\rho(\ell)\neq 1-\alpha$, and hence $\rho(\ell')\neq 0$ as required.
\end{proof}

We know by Proposition~\ref{abborders} that the shuffling delay is
necessarily longer than the longest Abelian unbordered prefix of
$x$. We will show that, in the case of self shuffling Sturmian
words, we can actually start the shuffle right after the longest
Abelian unbordered prefix.

\begin{lemma}\label{lem:Ab-lex}
Given a Sturmian  word $x$ of slope $\alpha<1/2$ and beginning in $0$ (resp. in $1$),
and $u$ a non-empty prefix of $x$, the following are equivalent:
\begin{itemize}
\item[(a)] $u$ is the longest Abelian border free prefix of $x$;
\item[(b)]  $u$ is the longest border free prefix of $x$;
\item[(c)]  $|u|$ is the smallest positive integer $n$ such that $T^n(x)<_{\lex}x$ (resp.  $T^n(x)>_{\lex}x$).
\end{itemize}
\end{lemma}

\begin{proof}
Unbordered factors of a Sturmian word are of the form $0B1$ or
$1B0$ where $B$ is a palindrome (see, e.g., \cite{HN04}).
Consequently, a Sturmian factor is unbordered if and only if it is
Abelian unbordered. Hence $(a)\Leftrightarrow (b)$.

$(a)\Leftrightarrow (c)$: Let us denote by $D$ the shuffling delay
of $x$, by $L$ the length of the longest (Abelian) unbordered
prefix of $x$, and by $N$ (resp. $M$) the smallest positive
integer $n$ such that $T^n(x)<_{\lex}x$ (resp. $T^n(x)>_{\lex}x$).
We know that $D\ge L$ by Proposition~\ref{abborders}. From the
proof of Theorem~\ref{the:rotation}, we also know that $D\le N$
(resp. $D\le M$) if $x$ begins in $0$ (resp. in $1$). Suppose $u$
is the longest unbordered prefix of $x$: $|u|=L$. We will show
$|u|=N$ whenever $x$ begins in $0$ and $|u|=M$ whenever $x$ begins
in $1$. First, suppose that  $x$ begins with the letter $0$.
 Then $|u|\le D\le N$. We will show $N\le |u|$.
Let $v$ be the shortest prefix of $x$ such that
$v^{-1}x<_{\lex}x$. Hence $|v|=N$. We claim that $v$ is
unbordered. Proceed by contradiction and suppose that $v$ is
bordered: $v=zs=pz$ for some non-trivial word $z$. Then we would
have  $z (v^{-1}x)=p^{-1}x>_{\lex}x=zs (v^{-1}x)$. But this would
imply $v^{-1}x>_{\lex} s (v^{-1}x)>_{\lex} x$, a contradiction
with the definition of $v$. Hence the claim follows and $|u|=N$.
Second, suppose that  $x$ begins with the letter $1$. Let $v$ be
the shortest prefix of $x$ such that $v^{-1}x>_{\lex}x$. Hence
$|v|=M$. Again we can show that $v$ is unbordered, and hence
$|u|=M$.
\end{proof}

\begin{proposition}
Let $x$ be a self-shuffling Sturmian word. Then the shuffling
delay of $x$ equals the length of the longest Abelian unbordered
prefix of $x$.
\end{proposition}

\begin{proof}
Follows from Lemma~\ref{lem:Ab-lex} and the proof of Theorem~\ref{the:rotation}.
\end{proof}

We are able to exhibit explicitly self-shuffles of the words $01C$ 
and $10 C$, where  $C$ is a charateristic Sturmian word.
These shuffles are described by the palindrome construction of Sturmian
words (see for instance \cite{del}).
Let $\rm{Pal}$ be the operator that maps a
finite word $w$ onto its {\em palindromic closure} $\rm{Pal}(w)$,
that is, the shortest palindrome having $w$ as a prefix.
Given an arbitrary binary sequence $(a_1,a_2,a_3,\ldots)$, called the {\em
directive sequence}, we can build a characteristic Sturmian word
by iterating the operator $\rm{\Phi}\colon\mathbb{N}\to \{0,1\}^*$
defined recursively by:
\[
\Phi(0)=\varepsilon \quad \text{ and } \quad \Phi(k)=
{\rm{Pal}}(\Phi(k-1)a_k) \text{ for } k\ge 1.
\]
Moreover, any characteristic Sturmian word may be obtained thanks to this construction.
For example, if the directive sequence is ${\bf d}=(0,0,1,0,1,1,0,1,\ldots)$,
we obtain the following characteristic Sturmian word:
\[
C=\lim_{k\to +\infty}\Phi(k)=
\hat{0}\hat{0}\hat{1}00\hat{0}100\hat{1}000100\hat{1}000100\hat{0}10010001001000100\hat{1}000\cdots
\]
To keep track of the directive sequence, we mark these letters by a ``hat".

Let us split the positive integers according to the fact that $a_k=0$ or $a_k=1$:
For all $i\ge1$, we define $k_0(i)$ (resp.  $k_1(i)$) to be the $i$-th positive integer $k$ such that $a_k=0$ (resp. $a_k=1$).
In the case of the directive sequence $\bf{d}$, we have
\[
(k_0(i))_{i\ge1}=(1,2,4,7,\ldots) \quad \text{ and } \quad (k_1(i))_{i\ge1}=(3,5,6,8,\ldots).
\]
For $k\ge 1$, we define the words $w_k$ by $\Phi(k)=\Phi(k-1)w_k$.

We can now describe the shuffles of $01C$ and $10C$:

\begin{proposition}
Suppose that  the directive sequence begins in $0$. Then
\begin{equation}\label{sh01C}
01C=01\prod_{k\ge1}w_k=01\prod_{i\ge 2}w_{k_0(i)}=0\prod_{i\ge 1}w_{k_1(i)}
\end{equation}
and
\begin{equation}\label{sh10C}
10C=10\prod_{k\ge1}w_k=10^{k_1(1)}\prod_{i\ge 2}w_{k_1(i)}=1\prod_{i\ge 1}w_{k_0(i)}.
\end{equation}
\end{proposition}

\begin{proof}
Note that we have
\[
w_{k_0(i)}=w_i=0  \quad \text{ for } 1\le i<k_1(1) \quad \text{ and }\quad w_{k_1(1)}=10^{k_1(1)-1}.
\]
This implies $01w_1=01w_{k_0(1)}=010\in\sh(01,0)$ and
\[
10\prod_{k=1}^{k_1(1)}w_k =10^{k_1(1)}w_{k_1(1)}
\in\sh\big(10^{k_1(1)},1\prod_{i=1}^{k_1(1)-1}w_{k_0(i)}\big)
=\sh\big(w_{k_1(1)}0,1\prod_{i=1}^{k_1(1)-1}w_i\big).
\]
Therefore it is sufficient to prove the equalities~\eqref{sh01C} and~\eqref{sh10C}.
In both cases, these equalities follow from the following
observation due to Risley and the fourth author \cite{Risley-Zamboni}:
For $a\in\{0,1\}$ and $i\ge 1$ such that $k_a(i)>k_1(1)$, we have
 \begin{equation}\label{hat-algorithm}
w_{k_a(i)}=\big(\Phi(k_a(i-1)-1) \big)^{-1} \Phi(k_a(i)-1).
\end{equation}
In other words, this means that, if the letter $a_k$ is equal to
$0$ (resp. to $1$),  the $k$-th iteration $\Phi(k)$ is obtained
from $\Phi(k-1)$ by concatenating to $\Phi(k-1)$ its suffix
starting from the last $\hat{0}$ (resp. from the last $\hat{1}$).
\end{proof}

For the words built on the directive sequence ${\bf d}$ we obtain:
\[
01C  =
(01)(\hat{0})(\underbrace{\hat{0}}_{w_2})(\underbrace{\hat{1}00}_{w_3})
(\underbrace{\hat{0}100}_{w_4})
(\underbrace{\hat{1}000100\hat{1}000100}_{w_5w_6})
(\underbrace{\hat{0}10010001001000100}_{w_7}) (\hat{1}000\cdots
\]
and
\[
10C  =(\underbrace{10\hat{0}\hat{0}}_{w_30})
(\underbrace{\hat{1}00\hat{0}100}_{1w_1w_2w_4})
(\underbrace{\hat{1}000100\hat{1}000100}_{w_5w_6})
(\underbrace{\hat{0}10010001001000100}_{w_7}) (\hat{1}000\cdots
\]\\

\begin{remark}It turns out that this shuffle is the same as the
one described by the general algorithm for shuffling Sturmian
words described in the proof of Theorem~\ref{the:rotation}.
We will show this fact in the case of $01C$.
The other case can be handled similarly.
Because we have assumed that the directive sequence starts with $0$,
we know that $0<\alpha<1/2$.
We start in Case 1.2 with $s=m=\ell=01C$:\\

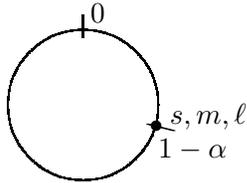
\unitlength=.5mm
\begin{picture}(50,60)
\put(50,0){
  \qbezier(25.000,10.000)(33.284,10.000)(39.142,15.858)
  \qbezier(39.142,15.858)(45.000,21.716)(45.000,30.000)
  \qbezier(45.000,30.000)(45.000,38.284)(39.142,44.142)
  \qbezier(39.142,44.142)(33.284,50.000)(25.000,50.000)
  \qbezier(25.000,50.000)(16.716,50.000)(10.858,44.142)
  \qbezier(10.858,44.142)( 5.000,38.284)(5.000,30.000)
  \qbezier( 5.000,30.000)( 5.000,21.716)(10.858,15.858)
  \qbezier(10.858,15.858)(16.716,10.000)(25.000,10.000)

  \put(44.5,24.5){\circle*{3}}  \put(48,25){$s,m,\ell$}

  \put(25,48){\line(0,1){6}}  \put(27,52){$0$}
  \put(42,25){\line(4,-1){7.5}} \put(45,16){$1-\alpha$}}

  \put(150,27.5){$s\le_{\lex} m =\ell$; $\neg C(s,m,\ell)$}
\end{picture}
\smallskip

 According to our general algorithm, $\ell$ follows $m$ until $0<\rho(m)=\rho(\ell)<\rho(s)$.
In this case, this means $\ell=m=C$.
We exchange the labels $s\leftrightarrow \ell$, hence $s=m=C$ and $\ell=01C$.
The first copy of $01C$ has output $01$.
We are now in Case 1.1 and so, $\ell$ follows $m$ as long as possible.
When it stops, we arrive in Case 2.1.
At this point, depending on the directive sequence,
the second copy has dished out $0\prod_{i=1}^{j}w_{k_1(i)}$ for some $j\ge0$.
For the next step, $s$ follows $m$ as long as possible,
and the first copy dishes out $\prod_{i=2}^{j}w_{k_0(i)}$ for some $j\ge2$.
When it stops, in principle, we arrive in Case 3.1, Case 3.2 or Case 4.
Let us show that, in the case we are concerned with, we necessarily arrive in
Case 3.1. Clearly, we cannot arrive in Case 3.2 because we started
with $s=m=\ell$ and the intercepts of $m$ and $\ell$ (resp. $m$
and $s$) cannot coincide more than once while the algorithm is performed.
The fact that we actually arrive in Case 3.1 follows
from~\eqref{hat-algorithm}: $m$ and $\ell$ coincide until $m$ sees $\hat{0}$
and $\ell$ sees $\hat{1}$, meaning that $m<_{\lex}\ell$.
Using the same kind of arguments, we can show that from Case 3.1, we necessarily arrive in Case 2.1.
Then, following the general algorithm, we simply alternate between Case 2.1 and Case 3.1.
At each step, this corresponds to dishing out either a product of $w_{k_0(i)}$ or a product of $w_{k_1(i)}$.
\end{remark}

We can also exhibit an explicit self-shuffle of the characteristic
Sturmian words. This shuffle is described in
Proposition~\ref{char}. We will need the following auxiliary
lemma. For a self-shuffling word $x=x_0x_1x_2\cdots$ we say that
letters $x_i$ and $x_j$ are congruent modulo its self-shuffle
defined by $N^1$ and $N^2$ (see Definition
\ref{k-self-shuffling}), if $i,j\in N^1$ or $i,j\in N^2$.

\begin{lemma}
Let $x\in \{0,1\}^\mathbb{N}$ be of the form $\prod_{i=1}^\infty
(0^{k_i}1)$ with $k_i\in\mathbb{N}$. Suppose that for each $n\geq 1$
\[
\sum_{i=1}^{n}k_i \leq \mbox{\rm{min}}\{\sum_{i=n+1}^{2n}k_i , \sum_{i=n+2}^{2n+1}k_i \}
\]
and for each $n\geq 2$
\[
 \sum_{i=1}^{n}k_i \geq \mbox{\rm{max}}\{\sum_{i=n+1}^{2n-1}k_i , \sum_{i=n+2}^{2n}k_i \}.
\]
Then $x$ is self-shuffling, and moreover there exists a
self-shuffle of $x$ such that no two consecutive $1$'s in $x$ are congruent modulo this shuffle.
\end{lemma}

\begin{proof}
For each $n\ge 1$, define
\[
u_n^1=\left\{
\begin{array}{ll}
k_1,& \text{ if }n\le 2\\
\displaystyle\sum_{i=1}^{n-1}k_i-\sum_{i=n+1}^{2n-2}k_i,& \text{ if }n\ge 3
\end{array}
\right.,\quad
v_n^1= \sum_{i=n+1}^{2n}k_i-\sum_{i=1}^{n}k_i
\]
and
\[
u_n^2=\left\{
\begin{array}{ll}
k_1,& \text{ if }n=1\\
\displaystyle\sum_{i=1}^ n k_i-\sum_{i=n+1}^{2n-1}k_i,& \text{ if }n\ge 2
\end{array}
\right.,\quad
v_n^2=\sum_{i=n+2}^{2n+1}k_i-\sum_{i=1}^{n}k_i.
\]
Then, for each $n\ge1$, we have
\[
u_n^1\ge 0,~v_n^1\ge 0,~u_n^2\ge 0,~v_n^2\ge 0
\]
by our assumption. Moreover, it is easy to see that
\[
x=\prod_{n=1}^\infty(0^{u_n^1}10^{v_n^1})
=\prod_{n=1}^\infty(0^{u_n^2}10^{v_n^2})
=\prod_{n=1}^\infty(0^{u_n^1}10^{v_n^1
})(0^{u_n^2}10^{v_n^2}).
\]
Thus, this self-shuffle of $x$ holds in a way that no two
consecutive $1$'s in $x$ are congruent modulo this shuffle.
\end{proof}

\begin{proposition}\label{char}
Let $x=\prod_{i=1}^\infty (0^{k_i}1)$ be a characteristic Sturmian word beginning in $0$.
Then $x$ satisfies each of the inequalities of the previous lemma and hence is self-shuffling.
\end{proposition}

\begin{proof}
We shall verify only the first inequality as the second is proved analogously.
Let $x=\prod_{i=1}^\infty (0^{k_i}1)$ be a characteristic Sturmian word.
We begin by observing that if $u$ is a prefix of $x$ ending in $0$, then $u$ is {\it rich} in $0$, i.e.,
there exists a factor $v$ of $x$ with $|u|=|v|$ such that $|u|_0=|v|_0+1$.
In fact we can take $v=1u0^{-1}$.
Similarly if $u$ ends in $1$ then $u$ is {\it poor} in $0$, i.e.,  rich in $1$.
Fix $n\geq1 $ and consider
the prefix $X=\prod_{i=1}^n(0^{k_i}1)$.
Then $X$ is poor in $0$.
Set
\[
Y= \prod _{i=n+1}^{2n} (0^{k_i}1)\,\,\,\,\,\mbox{\rm{and}}\,\,\,\,\,Z=\prod _{i=n+2}^{2n+1} (0^{k_i}1).
\]
We claim that
\[
|X|\leq \mbox{\rm{min}}\{|Y|, |Z|\}
\]
from which it follows that
\[
\sum_{i=1}^{n}k_i
= |X|_0 \leq \mbox{\rm{min}}\{|Y|_0, |Z_0\}
=\mbox{\rm{min}}\{\sum_{i=n+1}^{2n}k_i , \sum_{i=n+2}^{2n+1}k_i \} .
\]
In fact, suppose to the contrary that $|X|> \mbox{\rm{min}}\{|Y|,|Z|\}$;
note that $0X1^{-1}$, $1Y$ and $1Z$ are each factors of $w$ and
$|0X1^{-1}|\geq \mbox{\rm{min}}\{|1Y|, |1Z|\}$.
But $|0X1^{-1}|_1=n-1$ while $|1Y|_1=|1Z|_1=n+1$ contradicting that $x$ is balanced.
\end{proof}

As an almost immediate application of Corollary~\ref{ssSturm} we recover  the following  result
originally proved by Yasutomi in \cite{Ya} and later reproved by Berth\'e, Ei, Ito and Rao in \cite {BEIR}
and independently by Fagnot in \cite{Fa}.
We say an infinite word is {\it pure morphic} if it is a fixed point of some morphism different from the identity.

\begin{theorem}[Yasutomi \cite{Ya}]\label{Yasu}
Let $x\in \{0,1\}^\nats$ be a characteristic Sturmian word.
If $y$ is a pure morphic word in the orbit of $x$, then $y\in \{x, 0x, 1x, 01x,10x\}$.
\end{theorem}

\begin{proof}
We begin with some preliminary observations.
Let $\Omega(x)$ denote the set of all left and right infinite words $y$ such that $\mathcal{F}(x)=\mathcal{F}(y)$
where $\mathcal{F}(x)$ and $\mathcal{F}(y)$ denote the set of all factors of $x$ and $y$ respectively.
If $y\in \Omega(x)$ is a right infinite word, and $0y,1y\in \Omega(x)$, then $y=x$.
This is because every prefix  of $y$ is a left special factor and hence also a prefix of the characteristic word $x$.
Similarly if $y$ is a left infinite word and $y0,y1\in \Omega(x)$, then $y$ is equal to the reversal of $x$.
If $\tau$ is a morphism fixing some point $y\in \Omega(x)$, then $\tau(z)\in \Omega(x)$ for all $z\in \Omega(x)$.

Suppose to the contrary that $\tau\neq id$ is a substitution fixing a proper tail $y$ of $x$.
Then $y$ is self-shuffling by Corollary~\ref{ssSturm}.  Put $x=uy$ with $u\in \{0,1\}^+$.
Using the characterization of Sturmian morphisms (see Theorem 2.3.7 \& Lemma~2.3.13 in \cite{Lo2})
we deduce that $\tau$ must be primitive. Thus we can assume that $|\tau(a)|>1$ for each $a\in \{0,1\}$.
If $\tau(0)$ and $\tau(1)$ end in distinct letters, then as both $0\tau(x), 1\tau(x) \in \Omega(x)$, it follows that $\tau(x)=x$.
Since also $\tau(y)=y$ and $|\tau(u)|>|u|$, it follows that $y$ is a proper tail of itself, a contradiction since $x$ is aperiodic.
Thus $\tau(0)$ and $\tau(1)$ must end in the same letter.
Whence  by Corollary~\ref{primitive} it follows that every left extension of $y$ is self-shuffling,
which is again a contradiction since $0x$ and $1x$ are not self-shuffling.

Next suppose $\tau\neq id$ is a substitution fixing a point $y=uabx\in \Omega(x)$
where $u\in \{0,1\}^+$  and $\{a,b\}=\{0,1\}$.
Again we can suppose $\tau$ is primitive and $|\tau(0)|>1$ and $|\tau(1)|>1$.
If $\tau(0)$ and $\tau(1)$ begin in distinct letters, then $\tau(\tilde{x})0, \tau(\tilde{x})1\in \Omega(x)$
where $\tilde{x}$ denotes the reverse of $x$.
Thus $\tau(\tilde{x})=\tilde{x}$. Thus for each prefix $v$ of $abx$ we have $\tau(\tilde{x}v)=\tilde{x}\tau(v)$
whence $\tau(v)$ is also a prefix of $abx$. Hence $\tau(abx)=abx$.
As before this implies that $abx$ is a proper tail of itself which is a contradiction.
Thus  $\tau(0)$ and $\tau(1)$ begin in the same letter.
Whence  by Corollary~\ref{primitive} it follows that every tail of $y$ is self-shuffling,
which is again a contradiction since  $0x$ and $1x$ are not self-shuffling.
\end{proof}

\begin{remark}
In the case of the Fibonacci infinite word $x$, each of $\{x, 0x, 1x, 01x,10x\}$ is pure morphic.
For a general characteristic word $x$, since every point in the orbit of $x$ except for $0x$ and $1x$ is self-shuffling,
it follows that if $\tau$ is a morphism fixing $x$ (respectively $01x$ or $10x)$,
then $\tau(0)$ and $\tau(1)$ must end (respectively begin) in distinct letters.
\end{remark}

\section{Dynamical embedding and the stepping stone model}

Let $\A$ be a finite set with at least 2 elements.
For $k=2$ and $z\in\A^\nats$, let $G_z^k$ be the directed graph defined in Definition \ref{def_graph}.
We denote this $G_z^k$ by $G_z=(V_z,E_z)$.

We can sometimes embed the graph $G_z$ into a dynamical system nicely in the following sense.

\begin{definition}
Let $X$ be a compact metric space and $R$ be a continuous mapping from $X$ onto itself.
Let $x_0\in X$ and $K$ be a Borel subset of $X^2$.
We say that the quadruple $(X,R,x_0,K)$ is a {\em dynamical embedding} of the graph $G_z$
if $(i,j)\in V_z$ if and only if $(R^i x_0,R^j x_0)\in K$.
\end{definition}

\begin{definition}
Let $(X,R,x_0,K)$ be a dynamical embedding of the graph $G_z$.
The minimum subset $D$ of $X^2$ satisfying
\begin{enumerate}
\item $D\supset X^2\setminus K$, and
\item $(x,y)\in D$ if $(Rx,y)\in D$ and $(x,Ry)\in D$
\end{enumerate}
is called the {\it dead set}.
Let
\[
T=\{(x,y)\in K\colon \mbox{exactly one of }(Rx,y)\mbox{ or } (x,Ry)\mbox{ is in }D\}
\]
and $F=K\setminus(D\cup T)$.
We call $T$ the {\it deterministic set} and $F$ the {\it free set}.
\end{definition}

\begin{definition}
If $0<\alpha<1$ is an irrational number and $0\le\rho<1$ is any
real number, we let $z(\alpha,\rho)=z_0z_1z_2\cdots\in\D^\nats$
denote the Sturmian word defined by
\[
z_n=\lfloor (n+1)\alpha+\rho\rfloor-\lfloor n \alpha+\rho\rfloor.
\]
We also define $z(\alpha,1)$ to be the limit of $z(\alpha,\rho)$ as $\rho\to 1$.
That is, $z(\alpha,1)_0=1$ and $z(\alpha,1)_n=z(\alpha,0)_n$  for any $n\ge 1$.
\end{definition}

Remark that here $\alpha$ is the slope of the Sturmian word $z(\alpha,\rho)$
and if $\rho<1$, then $\rho$ is its intercept.
In Section~\ref{Sturmy}, we noted that the intercept $\rho(x)$ of a Sturmian word $x$ is not one-to-one.
So the intercept of $z(\alpha,1)$ is $0$ and not $1$.
This will not be confusing in any way in the following.

For $x',x''\in \mathbb{R}$, we will use the following notation:
\[
1_{x'\geq x''}=
\begin{cases}
1, & \mbox{if } x'\geq x'', \\
0,  & \mbox{if } x'< x''.
\end{cases}
\]

\begin{theorem}\label{embedding}
Let $z=z(\alpha,\rho)$.
Then the graph $G_z$ has a dynamical embedding $(\mathbb{T},R_\alpha,0,K)$,
where $\mathbb{T}=\mathbb{R}/\mathbb{Z}=[0,1)$,
$R_\alpha$ is the rotation by $\alpha$ on $\mathbb{T}$, and
\[
K=\{(x,y)\in[0,1)^2\colon 1_{x\ge1-\rho}+1_{y\ge1-\rho} =\lfloor x+y+\rho\rfloor\}.
\]
\end{theorem}

The set $K$ is illustrated on Figure \ref{Fig_K}.
\begin{figure}[htbp]
\setlength{\unitlength}{0.4mm}
\begin{center}
\begin{picture}(100,100)
\put(0,0){\line(1,0){101}}
\put(0,0){\line(0,1){101}}

\multiput(100,-1)(0,3){34}{$\cdot$}
\multiput(0,99)(3,0){34}{$\cdot$}
\multiput(0,30)(2,-2){17}{$\cdot$}
\multiput(34,32)(3,0){22}{$\cdot$}
\multiput(33,34)(0,3){22}{$\cdot$}
\put(34,0){\line(0,1){34}}
\put(0,34){\line(1,0){34}}

\put(0,2){\line(1,-1){2}}
\put(0,6){\line(1,-1){6}}
\put(0,10){\line(1,-1){10}}
\put(0,14){\line(1,-1){14}}
\put(0,18){\line(1,-1){18}}
\put(0,22){\line(1,-1){22}}
\put(0,26){\line(1,-1){26}}
\put(0,30){\line(1,-1){30}}

\put(36,100){\line(1,-1){64}}
\put(40,100){\line(1,-1){60}}
\put(44,100){\line(1,-1){56}}
\put(48,100){\line(1,-1){52}}
\put(52,100){\line(1,-1){48}}
\put(56,100){\line(1,-1){44}}
\put(60,100){\line(1,-1){40}}
\put(64,100){\line(1,-1){36}}
\put(68,100){\line(1,-1){32}}
\put(72,100){\line(1,-1){28}}
\put(76,100){\line(1,-1){24}}
\put(80,100){\line(1,-1){20}}
\put(84,100){\line(1,-1){16}}
\put(88,100){\line(1,-1){12}}
\put(92,100){\line(1,-1){9}}
\put(96,101){\line(1,-1){6}}
\put(98,101){\line(1,-1){4}}

\put(0,35){
\put(0,2){\line(1,-1){2}}
\put(0,6){\line(1,-1){6}}
\put(0,10){\line(1,-1){10}}
\put(0,14){\line(1,-1){14}}
\put(0,18){\line(1,-1){18}}
\put(0,22){\line(1,-1){22}}
\put(0,26){\line(1,-1){26}}
\put(0,30){\line(1,-1){30}}
}

\put(35,0){
\put(0,6){\line(1,-1){6}}
\put(0,10){\line(1,-1){10}}
\put(0,14){\line(1,-1){14}}
\put(0,18){\line(1,-1){18}}
\put(0,22){\line(1,-1){22}}
\put(0,26){\line(1,-1){26}}
\put(0,30){\line(1,-1){30}}
}

\multiput(0,69)(0,4){9}{\line(1,-1){34}}
\multiput(69,0)(4,0){9}{\line(-1,1){34}}

\put(5,100){\line(1,-1){29}}
\put(9,100){\line(1,-1){25}}
\put(13,100){\line(1,-1){21}}
\put(17,100){\line(1,-1){17}}
\put(21,100){\line(1,-1){13}}
\put(25,100){\line(1,-1){9}}
\put(29,100){\line(1,-1){6}}

\put(66,-66){
\put(5,100){\line(1,-1){29}}
\put(9,100){\line(1,-1){25}}
\put(13,100){\line(1,-1){21}}
\put(17,100){\line(1,-1){17}}
\put(21,100){\line(1,-1){13}}
\put(25,100){\line(1,-1){9}}
\put(29,100){\line(1,-1){6}}
}

\put(25,-10){$1-\rho$}
\put(-25,32){$1-\rho$}
\end{picture}
\end{center}
\caption{The set $K$ of Theorem~\ref{embedding}.}
\label{Fig_K}
\end{figure}
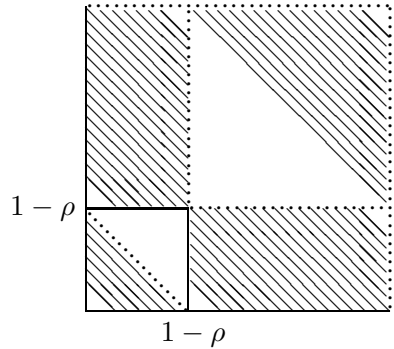

\begin{proof}
Let $G_z=(V_z,E_z)$.
Let $z(\alpha,\rho)=z_0z_1z_2\cdots$ with each $z_i\in\{0,1\}$. If
$0\le \rho<1$ then
\[
|z_0z_1\cdots z_{n-1}|_1= \sum_{i=0}^{n-1}z_i
=\sum_{i=0}^{n-1}(\lfloor (i+1)\alpha+\rho \rfloor-\lfloor i\alpha+\rho\rfloor)
=\lfloor n\alpha+\rho\rfloor.
\]
The same holds for $\rho=1$ since in this case,
\[
|z_0z_1\cdots z_{n-1}|_1= \sum_{i=0}^{n-1}z_i
=1+\sum_{i=1}^{n-1}(\lfloor (i+1)\alpha \rfloor-\lfloor i\alpha\rfloor)
=1+\lfloor n\alpha \rfloor
=\lfloor n\alpha+1\rfloor.
\]
Hence it holds that 
\begin{eqnarray*}
(i,j)\in V_z
&\Leftrightarrow& \lfloor (i+j)\alpha+\rho\rfloor=\lfloor i\alpha+\rho\rfloor+\lfloor j\alpha+\rho\rfloor\\
&\Leftrightarrow& \{(i+j)\alpha+\rho\}=\{i\alpha+\rho\}+\{j\alpha+\rho\}-\rho
\end{eqnarray*}
where $\{x\}$ means $x-\lfloor x\rfloor$.
Since
\[
\{(i+j)\alpha+\rho\}
=\{\{i\alpha\}+\{j\alpha\}+\rho\}
=\{i\alpha\}+\{j\alpha\}+\rho-\lfloor\{i\alpha\}+\{j\alpha\}+\rho\rfloor
\]
and, for each $k\in\{i,j\}$,
\[
\{k\alpha+\rho\}
=\{\{k\alpha\}+\rho\}
=\{k\alpha\}+\rho-1_{\{k\alpha\}\ge 1-\rho},
\]
we obtain that
\begin{eqnarray*}
(i,j)\in V_z
&\Leftrightarrow& 1_{\{i\alpha\}\ge 1-\rho}+1_{\{j\alpha\}\ge 1-\rho} =\lfloor\{i\alpha\}+\{j\alpha\}+\rho\rfloor\\
&\Leftrightarrow& (\{i\alpha\},\{j\alpha\})\in K\\
&\Leftrightarrow& (R_\alpha^i0,R_\alpha^j0)\in K.
\end{eqnarray*}
This implies that the quadruple $(\mathbb{T},R_\alpha,0,K)$ is a dynamical embedding of $G_z$.
\end{proof}

\begin{theorem}\label{stepping stone}
Let ${z}={z}(\alpha,\rho)$.
Then $z$ is self-shuffling if and only if there exists a sequence $(i_n,j_n)\in\nats^2$ such that
\begin{enumerate}
\item $i_n\le i_{n+1}$, $j_n\le j_{n+1}$, $i_n+j_n=n$ for any $n\in\nats$, \label{cond:1}\\
\item $\displaystyle{\lim_{n\to+\infty}}i_n=\displaystyle{\lim_{n\to+\infty}}j_n=\infty$, and \label{cond:2}\\
\item $1_{\{i_n\alpha\}\ge1-\rho}+1_{\{j_n\alpha\}\ge1-\rho}
=\lfloor\{i_n\alpha\}+\{j_n\alpha\}+\rho\rfloor$ for any $n\in\nats$.
\end{enumerate}
\end{theorem}

\begin{proof}
Clear from Theorems \ref{graph} and \ref{embedding}.
\end{proof}

Put $i\alpha=x,~j\alpha=y$ and consider the condition $(\{x\},\{y\}) \in K$,
where $K$ is defined as in Theorem~\ref{embedding}.
That is, $(x,y)$ is in $K+\mathbb{Z}^2$ (see Figures \ref{SS} and \ref{fig:rho=01}).
\begin{figure}[htbp]
\setlength{\unitlength}{0.4mm}
\begin{center}
\begin{picture}(150,170)
\put(-10,60){\line(1,0){180}}
\put(20,30){\line(0,1){180}}
\multiput(-40,0)(0,60){3}{\multiput(0,0)(60,0){3}
{\put(20,60){\thicklines\line(1,-1){40}}
\put(20,60){\thicklines\line(0,1){19}}
\put(79,20){\thicklines\line(-1,0){19}}
\put(20,64){\line(1,-1){44}}
\put(20,68){\line(1,-1){48}}
\put(20,72){\line(1,-1){52}}
\put(20,76){\line(1,-1){56}}
\put(21,79){\line(1,-1){58}}
\put(24,80){\line(1,-1){56}}
\put(28,80){\line(1,-1){52}}
\put(32,80){\line(1,-1){48}}
\put(36,80){\line(1,-1){44}}
\put(40,80){\line(1,-1){40}}
\put(44,80){\line(1,-1){36}}
\put(48,80){\line(1,-1){32}}
\put(52,80){\line(1,-1){28}}
\put(56,80){\line(1,-1){24}}
\multiput(22,78)(4,0){10}{$\cdot$}
\multiput(79,21)(0,4){10}{$\cdot$}
\multiput(61,76)(2,-2){9}{$\cdot$}
}}
\put(-40,0){
\put(60,60){\circle*{3}}
\put(84,60){\circle*{3}}
\put(108,60){\circle*{3}}
\put(132,60){\circle*{3}}
\put(132,84){\circle*{3}}
\put(132,108){\circle*{3}}
\put(156,108){\circle*{3}}
\put(180,108){\circle*{3}}
\put(180,132){\circle*{3}}
\put(180,156){\circle*{3}}
\put(180,180){\circle*{3}}
\put(61,52){\makebox(5,5){\bf 0}}
\put(85,52){\makebox(5,5){\bf 1}}
\put(109,52){\makebox(5,5){\bf 2}}
\put(133,52){\makebox(5,5){\bf 3}}
\put(133,80){\makebox(5,5){\bf 4}}
\put(133,100){\makebox(5,5){\bf 5}}
\put(157,100){\makebox(5,5){\bf 6}}
\put(181,100){\makebox(5,5){\bf 7}}
\put(181,124){\makebox(5,5){\bf 8}}
\put(181,148){\makebox(5,5){\bf 9}}
\put(181,172){\makebox(5,5){\bf 10}}}
\end{picture}
\end{center}
\vspace{-1cm}
\caption{The stepping stone for $0<\rho<1$ and a stepping stone
path: $(0,0)\to(\alpha,0)\to(2\alpha,0)\to(3\alpha,0)\to
(3\alpha,\alpha)\to(3\alpha,2\alpha)\to(4\alpha,2\alpha)\to(5\alpha,2\alpha)\to(5\alpha,3\alpha)
\to(5\alpha,4\alpha)\to(5\alpha,5\alpha)\to\cdots$}
\label{SS}
\end{figure}
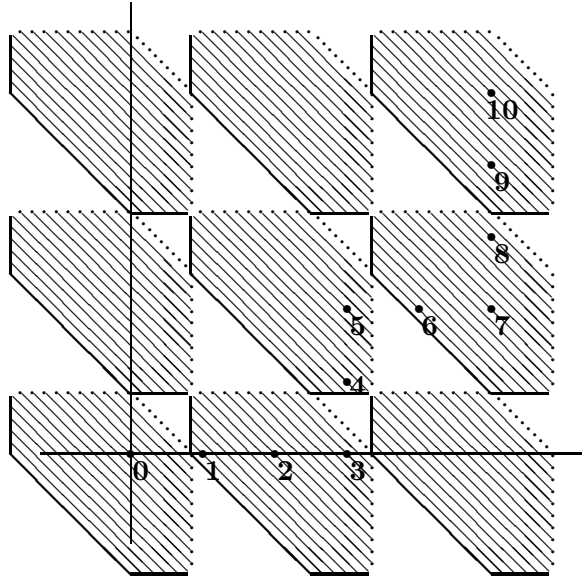

\begin{figure}[htbp]
\setlength{\unitlength}{0.3mm}
\begin{center}
\begin{picture}(150,60)
\put(-100,0){
\put(10,60){\line(1,0){150}}
\put(30,40){\line(0,1){150}}
\multiput(30,60)(0,60){2}
{\multiput(0,0)(60,0){3}{\put(0,0){\thicklines\line(0,1){60}}}}
\multiput(30,60)(0,60){3}{\multiput(0,0)(60,0){2}{
\put(0,0){\thicklines\line(1,0){60}}}}
\multiput(30,60)(0,60){2}{\multiput(0,0)(60,0){2}{
\multiput(0,56)(2,-2){29}{$\cdot$}
\put(4,54){\line(0,-1){54}}
\put(8,50){\line(0,-1){50}}
\put(12,46){\line(0,-1){46}}
\put(16,42){\line(0,-1){42}}
\put(20,38){\line(0,-1){38}}
\put(24,34){\line(0,-1){34}}
\put(28,30){\line(0,-1){30}}
\put(32,26){\line(0,-1){26}}
\put(36,22){\line(0,-1){22}}
\put(40,18){\line(0,-1){18}}
\put(44,14){\line(0,-1){14}}
\put(48,10){\line(0,-1){10}}
\put(52,6){\line(0,-1){6}}
\put(56,2){\line(0,-1){2}}}}}

\put(70,0){
\put(10,60){\line(1,0){150}}
\put(30,40){\line(0,1){150}}
\multiput(30,60)(0,60){2}{\multiput(0,0)(60,0){3}{
\put(0,0){\thicklines\line(0,1){60}}}}
\multiput(30,60)(0,60){3}{\multiput(0,0)(60,0){2}{
\put(0,0){\thicklines\line(1,0){60}}}}
\multiput(30,60)(0,60){2}{\multiput(0,0)(60,0){2}{
\multiput(0,56)(2,-2){29}{$\cdot$}
\put(4,60){\line(0,-1){2}}
\put(8,60){\line(0,-1){6}}
\put(12,60){\line(0,-1){10}}
\put(16,60){\line(0,-1){14}}
\put(20,60){\line(0,-1){18}}
\put(24,60){\line(0,-1){22}}
\put(28,60){\line(0,-1){26}}
\put(32,60){\line(0,-1){30}}
\put(36,60){\line(0,-1){34}}
\put(40,60){\line(0,-1){38}}
\put(44,60){\line(0,-1){42}}
\put(48,60){\line(0,-1){46}}
\put(52,60){\line(0,-1){50}}
\put(56,60){\line(0,-1){54}}}}}
\end{picture}
\end{center}
\vspace{-1cm}
\caption{The stepping stones for $\rho=0$ (left) and $\rho=1$ (right).}
\label{fig:rho=01}
\end{figure}
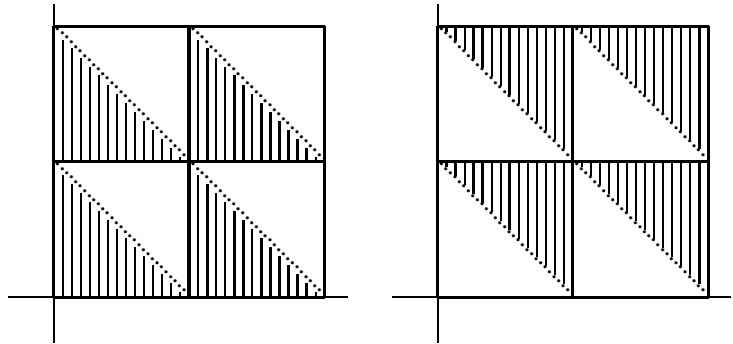
We call $K+\mathbb{Z}^2$ the {\it stepping stone} for $\rho$.
Hence, $z$ is self-shuffling if and only if there is a sequence $(i_n\alpha,j_n\alpha)\in K+\mathbb{Z}^2$
satisfying the conditions \eqref{cond:1} and \eqref{cond:2} of Theorem \ref{stepping stone}.
We call it a {\it stepping stone path} with respect to $(\alpha,\rho)$.
Therefore, by Theorem~\ref{the:rotation}, we have the following corollary.

\begin{corollary}
There exists a stepping stone path with respect to $(\alpha,\rho)$ if $0<\rho<1$.
\end{corollary}

Now we give another proof of the fact that Sturmian words of the form $aC$,
where $C$ is a characteristic Sturmian word, are not self-shuffling.

\begin{theorem}
Let $C$ be a characteristic Sturmian word.
Both $0C$ and $1C$ have no stepping stone path, and hence, are not self-shuffling.
\end{theorem}

\begin{proof}
We have $0C=z(\alpha,0)$ and $1C=z(\alpha,1)$.
Let $p_k/q_k$, with $k\ge1$, be the convergents of $\alpha$.

Assume first that $\rho=0$ and $z=z(\alpha,0)$.
By Theorem~\ref{embedding},
$(i,j)\in V_z$ if and only if $\{i\alpha\}+\{j\alpha\}<1$
since in this case,
$1_{\{i\alpha\}\ge 1-\rho}=1_{\{j\alpha\}\ge 1-\rho}=0$.
Since $\{q_{2k+1}\alpha\}$ is so close to 1,
we do not have $\{i\alpha\}+\{j\alpha\}<1$ for any $(i,j)\in\nats^2$
such that either $i=q_{2k+1}$ and $1\le j\le q_{2k+1}$ or $j=q_{2k+1}$ and $1\le i\le q_{2k+1}$.
Therefore, for each $k\ge 1$, no point in
\[
\{q_{2k+1}\}\times\{1,2,\ldots,q_{2k+1}\}\cup\{1,2,\ldots,q_{2k+1}\}\times\{q_{2k+1}\}
\]
belongs to the vertex set $V_{z}$.
Hence, there is no path in $G_z$ connecting $\vec{0}$ to $\vec{\infty}$.
Thus, $z(\alpha,0)$ is not self-shuffling.

We have a similar proof for $\rho=1$ and $z=z(\alpha,1)$. 
By Theorem~\ref{embedding},
$(i,j)\in V_z$ if and only if $\{i\alpha\}+\{j\alpha\}\ge1$
since in this case, $1_{\{i\alpha\}\ge 1-\rho}=1_{\{j\alpha\}\ge 1-\rho}=1$.
Since $\{q_{2k+2}\alpha\}$ is so close to 0, we have $\{i\alpha\}+\{j\alpha\}< 1$ 
 if either $i=q_{2k+2}$ and $1\le j\le q_{2k+2}$ or $j=q_{2k+2}$ and $1\le i\le q_{2k+2}$.
Hence, there is no path in $G_{z}$ connecting $\vec{0}$ to $\vec{\infty}$.
Thus, $z(\alpha,1)$ is not self-shuffling.
\end{proof}

Let $z=z(\alpha,\rho)$.
Let $(\mathbb{T},R_\alpha,0,K)$ be the dynamical embedding of the graph $G_{z}$ in Theorem~\ref{stepping stone}.
Let $z=z(\alpha,\rho)$.
Let $(\mathbb{T},R_\alpha,0,K)$ be the dynamical embedding of the graph $G_{z}$ in Theorem~\ref{stepping stone}.
We determine the dead set, the deterministic set and the free set in the easy case where
$(1-\rho)/2<\alpha<\min\{\rho,1-\rho\}$.
In fact, let
\begin{eqnarray*}
&D_1=\{(x,y)\in[0,1)^2\colon (R_\alpha x,y)\in \mathbb{T}^2\setminus K \},&\mbox{and}\\
&D_2=\{(x,y)\in[0,1)^2\colon (x,R_\alpha y)\in \mathbb{T}^2\setminus K \}.&
\end{eqnarray*}
Then, $D_1\cap D_2\cap K$ is obtained as in Figure~\ref{fig-D1D2K}.
\begin{figure}[htbp]
\setlength{\unitlength}{0.4mm}
\begin{center}
\begin{picture}(400,80)
\multiput(0,0)(0,100){2}{\thicklines\line(1,0){100}}
\multiput(0,0)(100,0){2}{\thicklines\line(0,1){100}}
\put(0,35){\line(1,-1){35}}
\put(35,100){\line(1,-1){65}}
\put(0,35){\line(1,0){100}}
\put(35,0){\line(0,1){100}}
\put(6,6){$K$}
\put(11,63){$K$}
\put(63,12){$K$}
\put(75,75){$K$}
\put(25,-10){$1-\rho$}

\put(135,0){
\put(0,0){\thicklines\line(0,1){15}}
\put(0,35){\thicklines\line(0,1){65}}
\put(100,0){\thicklines\line(0,1){100}}
\multiput(0,0)(0,100){2}{\thicklines\line(1,0){100}}
\put(-20,0){
\put(0,35){\line(1,-1){35}}
\put(35,100){\line(1,-1){65}}
\put(0,35){\line(1,0){100}}
\put(35,0){\line(0,1){100}}
\put(100,0){\line(0,1){100}}
\put(20,20){$D_1$}
\put(50,50){$D_1$}
\put(16,-10){$1-\rho-\alpha$}
\put(90,-10){$1-\alpha$}}}

\put(285,0){
\put(0,0){\thicklines\line(1,0){15}}
\put(35,0){\thicklines\line(1,0){65}}
\put(0,100){\thicklines\line(1,0){100}}
\multiput(0,0)(100,0){2}{\thicklines\line(0,1){100}}
\put(0,-20){
\put(0,35){\line(1,-1){35}}
\put(35,100){\line(1,-1){65}}
\put(0,35){\line(1,0){100}}
\put(35,0){\line(0,1){100}}
\put(0,100){\line(1,0){100}}
\put(20,20){$D_2$}
\put(50,50){$D_2$}
\put(-44,33){$1-\rho-\alpha$}
\put(-27,97){$1-\alpha$}}}
\end{picture}
\end{center}

\begin{center}
\begin{picture}(250,120)
\multiput(0,0)(0,100){2}{\thicklines\line(1,0){100}}
\multiput(0,0)(100,0){2}{\thicklines\line(0,1){100}}

\put(0,35){\line(1,-1){35}}
\put(35,100){\line(1,-1){65}}
\put(0,35){\line(1,0){100}}
\put(35,0){\line(0,1){100}}

\put(-20,0){
\put(0,35){\line(1,-1){35}}
\put(35,100){\line(1,-1){65}}
\put(0,35){\line(1,0){100}}
\put(35,0){\line(0,1){100}}
\multiput(0,0)(0,100){2}{\line(1,0){100}}
\multiput(0,0)(100,0){2}{\line(0,1){100}}}

\put(0,-20){
\put(0,35){\line(1,-1){35}}
\put(35,100){\line(1,-1){65}}
\put(0,35){\line(1,0){100}}
\put(35,0){\line(0,1){100}}
\multiput(0,0)(0,100){2}{\line(1,0){100}}
\multiput(0,0)(100,0){2}{\line(0,1){100}}}

\multiput(5,10)(2,0){4}{$\cdot$}
\multiput(7,8)(2,0){3}{$\cdot$}
\multiput(9,6)(2,0){2}{$\cdot$}
\multiput(11,4)(2,0){1}{$\cdot$}

\put(150,0){
\multiput(0,0)(0,100){2}{\thicklines\line(1,0){100}}
\multiput(0,0)(100,0){2}{\thicklines\line(0,1){100}}

\put(0,35){\line(1,-1){35}}
\multiput(35,96)(2,-2){33}{$\cdot$}
\put(35,35){\line(1,0){65}}
\put(35,35){\line(0,1){65}}
\multiput(2,32)(3,0){11}{$\cdot$}
\multiput(32,2)(0,3){11}{$\cdot$}

\put(0,15){\line(1,-1){15}}
\multiput(2,12)(3,0){4}{$\cdot$}
\multiput(13,2)(0,3){4}{$\cdot$}

\multiput(5,10)(2,0){4}{$\cdot$}
\multiput(7,8)(2,0){3}{$\cdot$}
\multiput(9,6)(2,0){2}{$\cdot$}
\multiput(11,4)(2,0){1}{$\cdot$}

\multiput(7,28)(2,0){11}{$\cdot$}
\multiput(10,25)(2,0){10}{$\cdot$}
\multiput(13,22)(2,0){8}{$\cdot$}
\multiput(16,19)(2,0){7}{$\cdot$}
\multiput(19,16)(2,0){6}{$\cdot$}
\multiput(22,13)(2,0){4}{$\cdot$}
\multiput(25,10)(2,0){2}{$\cdot$}
\multiput(28,7)(2,0){1}{$\cdot$}
\multiput(31,4)(2,0){1}{$\cdot$}

\multiput(36,93)(2,0){1}{$\cdot$}
\multiput(36,90)(2,0){2}{$\cdot$}
\multiput(36,87)(2,0){3}{$\cdot$}
\multiput(36,84)(2,0){5}{$\cdot$}
\multiput(36,81)(2,0){7}{$\cdot$}
\multiput(36,78)(2,0){8}{$\cdot$}
\multiput(36,75)(2,0){10}{$\cdot$}
\multiput(36,72)(2,0){11}{$\cdot$}
\multiput(36,69)(2,0){12}{$\cdot$}
\multiput(36,66)(2,0){14}{$\cdot$}
\multiput(36,63)(2,0){15}{$\cdot$}
\multiput(36,60)(2,0){17}{$\cdot$}
\multiput(36,57)(2,0){18}{$\cdot$}
\multiput(36,54)(2,0){20}{$\cdot$}
\multiput(36,51)(2,0){22}{$\cdot$}
\multiput(36,48)(2,0){23}{$\cdot$}
\multiput(36,45)(2,0){25}{$\cdot$}
\multiput(36,42)(2,0){26}{$\cdot$}
\multiput(36,39)(2,0){28}{$\cdot$}
\multiput(36,36)(2,0){29}{$\cdot$}
}
\end{picture}
\end{center}
\vspace{.5cm}
\caption{The sets $K,~D_1,~D_2$ (above) and
their intersection (below left),
$D$ (below right).}
\label{fig-D1D2K}
\end{figure}
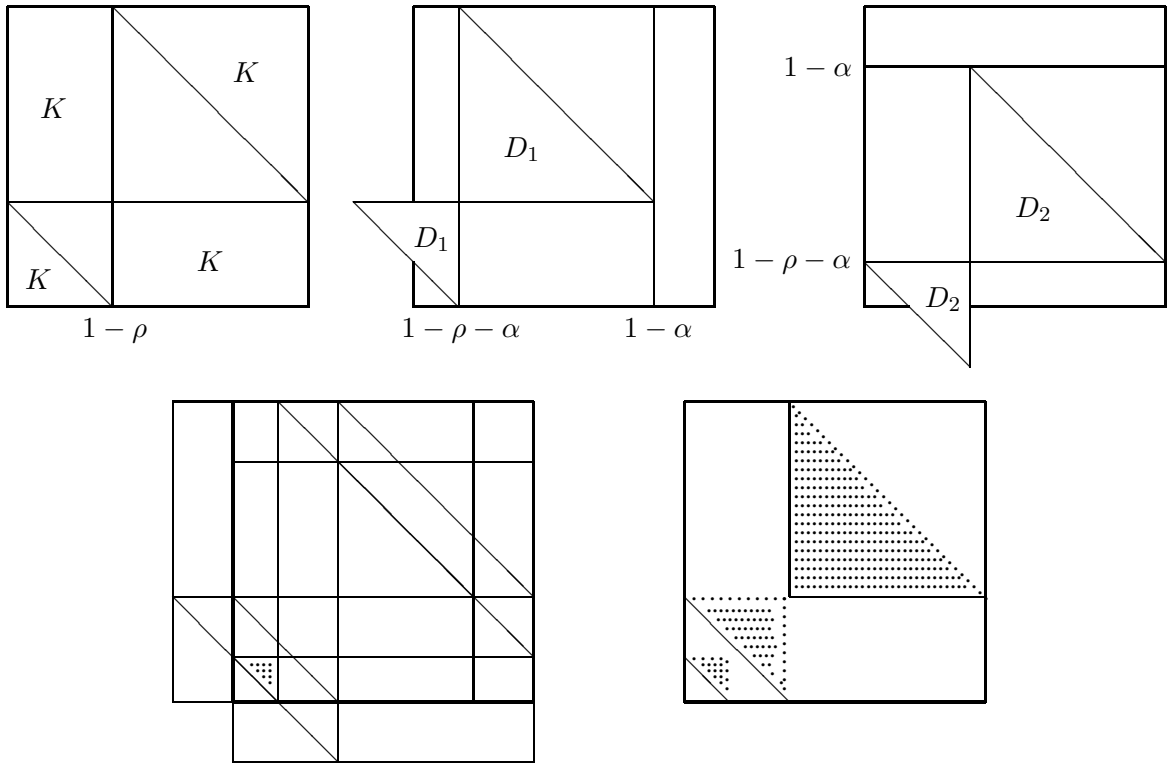

Therefore,
\[
D_1\cap D_2\cap K=\{(x,y)\in[0,1-\alpha-\rho)^2 \colon x+y\ge 1-\alpha-\rho\}.
\]
It is easily verified that:
\begin{enumerate}
\item If $(R_\alpha x,y)\in D_1\cap D_2\cap K$, then
$(x,R_\alpha y)\notin(D_1\cap D_2\cap K)\cup(\mathbb{T}^2
\setminus K)$.
\item If $(x,R_\alpha y)\in D_1\cap D_2\cap K$, then
$(R_\alpha x,y)\notin(D_1\cap D_2\cap K)
\cup(\mathbb{T}^2\setminus K)$.
\end{enumerate}
Hence, $D=(D_1\cap D_2\cap K)\cup(\mathbb{T}^2\setminus K)$ is the dead set.

Let
$$T_1=\{(x,y)\in\mathbb{T}^2\setminus D:~(x,R_\alpha y
)\in D\}\mbox{ and }T_2=\{(x,y)\in\mathbb{T}^2\setminus D:
~(R_\alpha x,y)\in D\}.$$
Then,
\begin{multline*}
T_1=\{(x,y)\in(1-\rho-\alpha,1-\rho]\times[0,1-\rho-\alpha):~x+y<1-\rho\}\\
\cup\{(x,y)\in[1-\rho,1)\times[1-\rho-\alpha,1-\rho):~x+y<2-\rho-\alpha\}\\
\cup\{(x,y)\in[0,1-\rho)\times[1-\rho,1)\colon x+y\ge 2-2\alpha-\rho\}
\end{multline*}
holds. If $(x,y)\in T_1$, then $(R_\alpha x,y)
\in\mathbb{T}^2\setminus D$ since otherwise, $(x,y)\in D$ by the definition of $D$.
Hence,
$(R_\alpha x,y)\in\mathbb{T}^2\setminus D$ and
$(x,R_\alpha y)\in D$ holds if $(x,y)\in T_1$.
The same things hold for $T_2$ in the symmetrical sense.
 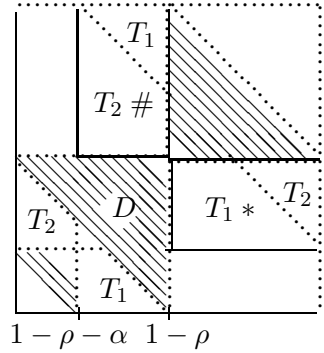
\begin{figure}[htbp]
\setlength{\unitlength}{0.4mm}
\begin{center}
\begin{picture}(100,110)
\put(0,0){\line(1,0){100}}
\put(0,0){\line(0,1){100}}
\multiput(100,0)(0,3){34}{$\cdot$}
\multiput(0,100)(3,0){34}{$\cdot$}
\put(0,52){\line(1,-1){50}}
\put(51,51){\line(0,1){50}}
\put(51,51){\line(1,0){50}}
\multiput(0,50)(3,0){17}{$\cdot$}
\multiput(50,0)(0,3){17}{$\cdot$}
\multiput(50,100)(2,-2){25}{$\cdot$}
\put(0,20){\line(1,-1){20}}
\multiput(0,19)(3,0){10}{$\cdot$}
\multiput(19,0)(0,3){10}{$\cdot$}
\put(4,20){\line(1,-1){16}}
\put(8,20){\line(1,-1){12}}
\put(12,20){\line(1,-1){8}}
\put(15,21){\line(1,-1){6}}
\put(4,52){\line(1,-1){46}}
\put(8,52){\line(1,-1){42}}
\put(12,52){\line(1,-1){16}}
\put(36,28){\line(1,-1){14}}
\put(16,52){\line(1,-1){14}}
\put(38,30){\line(1,-1){12}}
\put(20,52){\line(1,-1){12}}
\put(40,32){\line(1,-1){10}}

\put(24,52){\line(1,-1){26}}
\put(28,52){\line(1,-1){22}}
\put(32,52){\line(1,-1){18}}
\put(36,52){\line(1,-1){14}}
\put(40,52){\line(1,-1){10}}
\put(44,52){\line(1,-1){6}}
\put(51,57){\line(1,-1){6}}
\put(51,61){\line(1,-1){10}}
\put(51,65){\line(1,-1){14}}
\put(51,69){\line(1,-1){18}}
\put(51,73){\line(1,-1){22}}
\put(51,77){\line(1,-1){26}}
\put(51,81){\line(1,-1){30}}
\put(51,85){\line(1,-1){34}}
\put(51,89){\line(1,-1){38}}
\put(51,93){\line(1,-1){42}}
\put(51,97){\line(1,-1){46}}
\put(50,21){\line(1,0){50}}
\put(52,21){\line(0,1){30}}
\multiput(50,48)(3,0){17}{$\cdot$}
\multiput(70,50)(2,-2){15}{$\cdot$}
\put(20.5,52){\line(0,1){49}}
\put(20.5,52){\line(1,0){30}}
\multiput(49,50)(0,3){17}{$\cdot$}
\multiput(50,70)(-2,2){15}{$\cdot$}
\multiput(0,47)(2,-2){10}{$\cdot$}
\multiput(30,17)(2,-2){10}{$\cdot$}
\put(32,32){$D$}
\put(63,32){$T_1$}
\put(75,32){$*$}
\put(37,90){$T_1$}
\put(26,67){$T_2$}
\put(38,67){$\#$}
\put(89,37){$T_2$}
\put(28,5){$T_1$}
\put(4,28){$T_2$}
\put(21,2){\line(0,-1){4}}
\put(-2,-10){$1-\rho-\alpha$}
\put(51,2){\line(0,-1){4}}
\put(43,-10){$1-\rho$}
\end{picture}
\end{center}
\caption{The dead set $D$ together with $T_1\cup T_2$.}
\label{fig:D-T1-T2}
\end{figure}

Let $T=T_1\cup T_2$ and $F=K\setminus(D\cup T)$.
By the definition, it holds that $(R_\alpha x,y)\notin D$
and $(x,R_\alpha y)\notin D$ for any $(x,y)\in F$.
For each $i\in\{1,2\}$, define a mapping $\tilde{R}_{\alpha,i}\colon F\to F$.
For $(x,y)\in\mathbb{T}^2$, we denote
\[
R_{\alpha,i}(x,y)=\left\{\begin{array}{ll}
(R_\alpha x,y),&\mbox{if }i=1\\
(x, R_\alpha y),&\mbox{if }i=2.
\end{array}\right.
\]
Take an arbitrary $(x_0,y_0)\in F$.
Let $(x_1,y_1)=R_{\alpha,i}(x_0,y_0)$.
If $(x_1,y_1)\in F$, then let $\tilde{R}_{\alpha,i}(x_0,y_0)=(x_1,y_1)$.
If $(x_1,y_1)\notin F$, then either $(x_1,y_1)\in T_1$ or $(x_1,y_1)\in T_2$.
Let $(x_2,y_2)=R_{\alpha,1}(x_1,y_1)$ in the former case, and let $(x_2,y_2)=R_{\alpha,2}(x_1,y_1)$ in the latter case.
If $(x_2,y_2)\in F$, then let $\tilde{R}_{\alpha,i}(x_0,y_0)=(x_2,y_2)$.
Repeat this procedure until we get $(x_n,y_n)\in F$.
Then, we define $\tilde{R}_{\alpha,i}(x_0,y_0)=(x_n,y_n)$.
If $(x_n,y_n)\notin F$ for any $n\ge 1$, then
$\tilde{R}_{\alpha,i}(x_0,y_0)$ is not defined, which never happens in our case.
This is easily seen from Figure~\ref{fig:D-T1-T2}.
That is, if $(x,y)\in T_1$, then $R_{\alpha,1}(x,y)\in F$ except for the case when $(x,y)$ is in the $*$-marked region.
If $(x,y)$ is in the $*$-marked region, then take the first $n>0$ such that $R_{\alpha,1}^n(x,y)\notin T_1$.
Then, either $R_{\alpha,1}^n(x,y)\in F$ or $R_{\alpha,1}^n(x,y)\in T_2$.
In the latter case, $R_{\alpha,1}^n(x,y)$ is not in the $\#$-marked region so that $R_{\alpha,2}R_{\alpha,1}^n(x,y)\in F$.
The same for the case $(x,y)\in T_2$.
Therefore, the infinite paths starting $(0,0)$ in $G_z$
correspond bijectively to the infinite sequences of mappings applied to $(0,0)$
\[
    \cdots\tilde{R}_{\alpha,i_3}\tilde{R}_{\alpha,i_2}\tilde{R}_{\alpha,i_1} (0,0)
\]
with $i_1,i_2,\ldots\in\{1,2\}$.
Note that both of $\tilde{R}_{\alpha,1}$ and $\tilde{R}_{\alpha,2}$ are
domain exchange transformations on $F$.

\section{Open questions}




Typically a self-shuffling word can be shuffled in more than one way, i.e., it defines several different steering words.
One may ask:

\begin{question}
Does there exist a self-shuffling word admitting a unique steering word, i.e.,
which can be self-shuffled to produce itself in one and only one way?
\end{question}

We saw that every aperiodic uniformly recurrent word contains an element in its shift orbit closure which is not self-shuffling.

\begin{question}
Does there exist a word $x\in \A^\nats$ for which no element of its shift orbit closure  is self-shuffling?
\end{question}


\begin{thebibliography}{50}

\bibitem{AS1}  J.-P. Allouche, J. Shallit, The ubiquitous Prouhet-Thue-Morse sequence.
In: Sequences and their applications, Proceedings of SETA'98, C. Ding, T. Helleseth and H.
Niederreiter (Eds.) (1999), Springer Verlag, 1--16.

\bibitem{AS2}  J.-P. Allouche, J. Shallit, Automatic sequences. Theory, applications, generalizations.
Cambridge University Press, 2003.

\bibitem{berstel} J. Berstel, A rewriting of Fife's theorem about overlap-free
words, in J. Karhum\"{a}ki, H. Maurer, G. Rozenberg, eds., Results
and Trends in Theoretical Computer Science, Lecture Notes in
Computer Science 812, Springer-Verlag, 1994, 19--29.

\bibitem{BEIR} V. Berth\'e, H. Ei, S. Ito,  H. Rao, On substitution invariant Sturmian words: an application of Rauzy fractals.
Theor. Inform. Appl. 41 (2007),  329--349.

\bibitem{BS} S. Buss, M. Soltys, {\em Unshuffling a square is
NP-hard}, arXiv:1211.7161

\bibitem{CaK} J. Cassaigne, J.  Karhum\"{a}ki,  Toeplitz Words, Generalized
Periodicity and Periodically Iterated Morphisms. European J.
Combin. 18 (1997),  497--510.

\bibitem{icalp} \'E. Charlier, T. Kamae, S. Puzynina, L. Q. Zamboni {\it
Self-shuffling words.} ICALP 2013, Part II, LNCS 7966 (2013),
113--124.

\bibitem{del} A. de Luca,
\newblock  Sturmian words: structure, combinatorics, and their arithmetics,
\newblock Theoret. Comput. Sci. 183 (1997),  45--82.

\bibitem{Du} F. Durand,
\newblock  A characterization of substitutive sequences using return words,
\newblock Discrete Math. 179 (1998),  89--101.


\bibitem{HN04} T. Harju, D. Nowotka, Minimal Duval extensions,
Internat. J. Foundations Comput. Sci., 15 (2004), 349--354.



\bibitem{HRS} D. Henshall, N. Rampersad, J.  Shallit,  Shuffling and
unshuffling. Bull. EATCS, 107 (2012), 131--142.

\bibitem{Fa} I. Fagnot, A little more about morphic Sturmian words,  Theor. Inform. Appl. 40 (2006),  511--518.


\bibitem{Lo2} M. Lothaire, Algebraic Combinatorics On Words,
vol.~90 of Encyclopedia of Mathematics and its Applications,
Cambridge University Press, U.K., 2002.



\bibitem{MoHe2} M. Morse, G.A. Hedlund,  Symbolic dynamics II:
Sturmian sequences. Amer. J. Math. 62 (1940),  1--42.

\bibitem{Risley-Zamboni} R. Risley, L.Q.  Zamboni, A generalization of {S}turmian sequences: combinatorial
              structure and transcendence. Acta Arith., 95 (2000),  167--184.

\bibitem{RV} R. Rizzi and S. Vialette, {\em On recognizing words that are squares for the shuffle product},
LNCS 7913 (CSR 2013), 235--245.

\bibitem{SI} R. Siromoney,  L. Mathew, V.R.  Dare, K.G. Subramanian,  Infinite Lyndon words, Inform. Process. Lett. 50 (1994), 101--104


\bibitem{Th1} A. Thue, \"Uber unendliche Zeichenreihen, Norske Vid.\ Selsk.\
 Skr.\ I Math-Nat. Kl. 7 (1906),  1--22.


\bibitem{Ya} S.-I. Yasutomi,  On sturmian sequences which are invariant under some substitutions.
In: Number theory and its applications, Proceedings of the
conference held at the RIMS, Kyoto, Japan, November 10--14, 1997,
edited by Kanemitsu, Shigeru et al. Kluwer Acad. Publ. Dordrecht
(1999), 347--373.

\end{thebibliography}
\end{document}